\documentclass[a4paper,10pt]{article}

\usepackage{amssymb}
\usepackage{graphicx}
\usepackage{bbm}
\usepackage{mathrsfs}
\usepackage{amsmath}
\usepackage{amsthm}
\usepackage[center]{caption}
\usepackage{enumerate}
\usepackage{verbatim}
\usepackage{authblk}

\theoremstyle{plain}
\newtheorem{thm}{Theorem}[]
\newtheorem{cor}[thm]{Corollary}
\newtheorem{lem}[thm]{Lemma}
\newtheorem{prop}[thm]{Proposition}

\theoremstyle{definition}
\newtheorem{ex}{Example}
\newtheorem{defn}[thm]{Definition}
\newtheorem*{rmk}{Remark}

\newtheorem*{mainprf}{Proof of Theorem \ref{mainthm}}

\newcommand{\Qb}{\mathbb{Q}}
\newcommand{\Pb}{\mathbb{P}}
\newcommand{\Rb}{\mathbb{R}}
\newcommand{\Pt}{\tilde{\mathbb{P}}}
\newcommand{\Qt}{\tilde{\mathbb{Q}}}

\newcommand{\Fg}{\mathcal{F}}
\newcommand{\Gg}{\mathcal{G}}
\newcommand{\Ft}{\tilde{\mathcal{F}}}

\newcommand{\hs}{\hspace{2mm}}
\newcommand{\hsl}{\hspace{1mm}}
\newcommand{\ind}{\mathbbm{1}}

\author[1]{S.C.~Harris}

\author[2]{M.I.~Roberts}

\affil[1]{Department of Mathematical Sciences, University of Bath, Claverton Down, Bath, BA2 7AY. \emph{E-mail:} S.C.Harris@bath.ac.uk}
\affil[2]{Laboratoire de Probabilit\'es et Mod\`eles Al\'eatoires, Universit\'e Paris VI, 175 rue du Chevaleret, 75013 Paris. \emph{E-mail:} matthew.roberts@upmc.fr}

\title{The unscaled paths of\\ branching Brownian motion}

\begin{document}

\maketitle

\subsection*{Abstract}
For a set $A\subset C[0,\infty)$, we give new results on the growth of the number of particles in a branching Brownian motion whose paths fall within $A$. We show that it is possible to work without rescaling the paths. We give large deviations probabilities as well as a more sophisticated proof of a result on growth in the number of particles along certain sets of paths. Our results reveal that the number of particles can oscillate dramatically. We also obtain new results on the number of particles near the frontier of the model. The methods used are entirely probabilistic.

\section{Introduction}

\noindent
One of the most natural questions to ask about branching Brownian motion (BBM) concerns the position of the extremal particle --- the particle with maximal position at each time $t\geq0$. It is well-known that its speed --- its position divided by time --- converges almost surely to $\sqrt{2r}$ as $t\to\infty$. In fact, far more precise results are available, such as that given by Bramson \cite{bramson:maximal_displacement_BBM} via some powerful and explicit analysis of the Brownian bridge.

Once we know the speed of the extremal particle at large times $t$, we might ask about its history: have its ancestors stayed close to the critical speed throughout, or have they hovered around in the mass of particles near the origin and made a late dash as we get close to time $t$? One way of interpreting this question is to consider branching Brownian motion with absorption. One imagines an absorbing line $L(t)= -x + \gamma t$ where $\gamma$ is a constant close to the critical value $\sqrt{2r}$, such that whenever a particle hits the line $L(t)$ it disappears and is removed from the system. Are there any particles still present at large times? If so then we may consider them to have stayed ``close'' to the extremal edge of the system.

This model for BBM with killing on the line was studied by Kesten \cite{kesten:BBM_with_absorption}, who discovered asymptotics for extinction probabilities and numbers of particles in intervals of the area above the absorbing line. To choose two examples of particular interest, Kesten shows that if $\gamma<\sqrt{2r}$ then there is strictly positive probability that $N(t)$ never becomes empty; and that in the critical case $\gamma = \sqrt{2r}$, the probability that there is at least one particle present at time $t$ is approximately $\exp(-k t^{1/3})$ for some positive constant $k$. Thus it is possible for particles to stay above the line $\Gamma(t) = -x + \gamma t$ for all time whenever $\gamma < \sqrt{2r}$, and that this is not the case when $\gamma=\sqrt{2r}$. 
Our next question might be: can particles stay within $t^{\beta}$ (plus a constant, say) of the critical line for $\beta\in(0,1)$? Indeed, we could also attempt to generalise by moving away from the critical line --- given a path $f:[0,\infty)\to\Rb$, are there particles that stay close to $f$ and, if so, how close? Such questions provide motivation for this article.

\vspace{3mm}

The classical \emph{scaled} path properties of branching Brownian motion (BBM) have now been well-studied: for example, see Lee \cite{lee:large_deviations_for_branching_diffusions} and Hardy and Harris \cite{hardy_harris:spine_bbm_large_deviations} for large deviation results on ``difficult'' paths which have a small probability of any particle following them, and Git \cite{git:almost_sure_path_properties} and Harris and Roberts \cite{harris_roberts:scaled_growth} for the almost sure growth rate of the number of particles near ``easy'' paths along which we see exponential growth in the number of particles. To give these results, the paths of a BBM are rescaled onto the interval $[0,1]$, echoing the approach of Schilder's theorem for a single Brownian motion.

In this article, as suggested above, we consider a problem similar in theme, but from a more naive viewpoint. We are given a fixed set of paths $A\subset C[0,\infty)$ and we want to know how many particles in a BBM have paths within this set $A$. Similar problems in the case of a single Brownian motion have been considered by Kesten \cite{kesten:BBM_with_absorption} and Novikov \cite{novikov:asymptotic_behaviour_nonexit_probs}. The simplest case is to consider the ball $B(f,L)$ of fixed width $L>0$ about a single continuous path $f:[0,\infty)\to\Rb$, (we will, however, consider more general sets of paths). Clearly there is a positive probability that no particle will stay within this fixed ``tube'' --- indeed, the very first particle could wander away from $f$ before it has the chance to give birth to another --- and in this event we say that the process becomes extinct.

The intuition is that the growth of the population due to branching is in constant competition with the ``deaths'' due to particles failing to follow the function $f$. Thus a natural condition arises: if the gradient of $f$ is too large, then the process eventually dies out almost surely and we may ask for the large deviation probabilities of survival up to large times; otherwise, if the gradient of $f$ remains sufficiently small, then we may condition on non-extinction and give an almost sure result on the number of particles along the path.

One payoff for our less classical approach is that we immediately see a dramatic oscillation in the number of particles along certain paths. This unusual behaviour (not seen in the existing literature) has a simple explanation which we demonstrate via some illuminating examples in Section \ref{examples_section}.

\vspace{3mm}

In our proofs, we take advantage of spine techniques to interpret the change of measure given by a carefully chosen martingale. The spine tools give us an intuitive probabilistic handle on the problem, without which we would certainly need substantial extra technical work in several areas. Our particular change of measure involves forcing one particle (the spine) to stay within a tube of varying radius $L(t)$, $t\geq0$ about a function $f$. This change of measure is the result of a new martingale which we develop in Section \ref{spine_section}. We then use the spine decomposition first introduced by Lyons et al. \cite{lyons_et_al:conceptual_llogl_mean_behaviour_bps}, which allows us to bound the growth of the system by looking at the births along the spine.

Even with the spine theory the problem retains significant difficulty inherent in its time-inhomogeneity. This fact is underlined by the observation that even in the case $A=B(f,L)$ we are essentially considering a one-dimensional branching diffusion with time-dependent drift, and asking how many particles remain within a bounded domain about the origin. It turns out that the main difficulty is in showing that extinction of the process coincides (to within a null set) with the event that the limit of our martingale is zero. Standard tools -- analytic or probabilistic -- cannot be applied; instead we proceed by our own methods in Section \ref{inv_spine}, using in particular an identity from Harris and Roberts \cite{harris_roberts:extinction_letter}.

For simplicity, we consider only standard one-dimensional binary branching Brownian motion, but we note that our work could be extended to a wide range of other branching diffusions. In particular the spine methods are well-suited to the situation where each particle gives birth to a random number of new particles, and methods similar to those used in the original papers of Lyons et al \cite{kurtz_et_al:conceptual_kesten_stigum, lyons:simple_path_to_biggins, lyons_et_al:conceptual_llogl_mean_behaviour_bps} could be used to extend our result.

Our main theorem concerns only sets of paths away from criticality. However, by adapting the methods from the proof of this theorem, we are able to obtain new results on the number of particles near the extremes of the system (see Theorems \ref{betaex} and \ref{thirdex}). These results answer the questions raised in the above discussion and, as was mentioned there, should be compared to the work of Bramson \cite{bramson:maximal_displacement_BBM} on the position of the right-most particle, and of Kesten \cite{kesten:BBM_with_absorption} and other authors on BBM with absorption.

\section{Main results}

\subsection{Initial definitions}\label{defns}
We consider a branching Brownian motion starting with one particle at the origin, whereby each particle moves independently and undergoes independent dyadic branching at exponential rate $r>0$. We let the set of particles alive at time $t$ be $N(t)$, and for each particle $u \in N(t)$ denote its position at time $t$ by $X_u(t)$. We extend this notion of a particle's position to include the positions of its ancestors; that is, if $u\in N(t)$ has ancestor $v\in N(s)$ for some $s<t$, then we set $X_u(s):=X_v(s)$. This setup will be given in more detail in Section \ref{spine_section}.

Fix a continuous function $f:[0,\infty)\rightarrow\Rb$, and another $L:[0,\infty)\rightarrow(0,\infty)$. If $f$ and $L$ are twice continuously differentiable then we define
\[E(t):= |f'(t)|L(t) + \int_0^t |f''(s)|L(s)ds + \frac{1}{2}|L'(t)|L(t) + \frac{1}{2}\int_0^t |L''(s)|L(s) ds\]
and
\[S := \liminf_{t\to\infty} \frac{1}{t}\int_0^t \left(r - \frac{1}{2} f'(s)^2 - \frac{\pi^2}{8L(s)^2} + \frac{L'(s)}{2L(s)}\right) ds.\]
We say that the pair $(f,L)$ satisfies the \emph{usual conditions} if:
\renewcommand{\labelenumi}{(\arabic{enumi})}
\begin{enumerate}[(I)]
\item $f(0)=0$;
\item $f$ and $L$ are twice continuously differentiable;
\item $\lim_{t\to\infty}E(t)/t = 0$;
\item $S\in(-\infty,\infty)$.
\end{enumerate}
We assume throughout this article that, unless otherwise stated, these conditions hold. We consider initially the class of sets of the form
\[B(f,L) := \{g\in C[0,\infty): |g(t)-f(t)| < L(t) \hs \forall t\in[0,\infty)\}\]
such that $f$ and $L$ satisfy the usual conditions. After we obtain our results we will be able to extend them in a natural way to cover more general subsets of $C[0,\infty)$ --- see Section \ref{extensions} --- but for now these conditions will allow us to apply integration by parts theorems without any complications. Although condition (III) may appear unnatural, there are clear reasons behind it, some of which are demonstrated via example in Section \ref{extensions}. There are also similar conditions in the work on a single Brownian motion by Kesten \cite{kesten:BBM_with_absorption} and Novikov \cite{novikov:asymptotic_behaviour_nonexit_probs}.

Define
\[\hat N(t) := \left\{u \in N(t) : |X_u(s)-f(s)|< L(s) \hs \forall s\leq t\right\},\]
the set of particles that have stayed within distance $L$ of the function $f$ for all times $s\leq t$. We wish to study the number of particles in $\hat{N}(t)$ at large times. Let
\[\Upsilon := \inf\{t\geq0 : \hat N(t)=\emptyset\}.\]
We call $\Upsilon$ the \emph{extinction time} for the process, and say that the process has become \emph{extinct} by time $t$ if $\Upsilon\leq t$. When we talk about \emph{survival} or \emph{non-extinction}, we mean the event $\Upsilon = \infty$.

\subsection{The non-critical case, $S\neq 0$}\label{mainsec}
We now state our main result in the non-critical case when $S\neq0$. Most of this article will be concerned with proving this theorem.
\begin{thm}\label{mainthm}
If $S < 0$, then $\Upsilon<\infty$ almost surely and
\[\frac{\log \Pb(\hat N(t) \neq \emptyset)}{\inf_{s\leq t}\int_0^s \left(r - \frac{1}{2}f'(u)^2 - \frac{\pi^2}{8L(u)^2} + \frac{L'(u)}{2L(u)}\right)du} \hs\longrightarrow\hs 1.\]
On the other hand, if $S > 0$, then $\Pb(\Upsilon=\infty)>0$ and almost surely on survival we have
\[\frac{\log|\hat N(t)|}{\int_0^t \left(r - \frac{1}{2} f'(s)^2 - \frac{\pi^2}{8L(s)^2} + \frac{L'(s)}{2L(s)}\right) ds} \hs\longrightarrow\hs 1.\]
\end{thm}

\noindent
As mentioned earlier, this theorem can be extended to cover more general sets, and we give results in this direction in Section \ref{extensions}. The behaviour at criticality (S=0) depends on the finer behaviour of $f$ and $L$, but we are able to give some results in particular important cases in Section \ref{beta_sec} below. We note the following corollary, which is easily deduced from Theorem \ref{mainthm}.

\begin{cor}
If $S > 0$, then almost surely on survival we have
\[\limsup_{t\to\infty}\frac{1}{t}\log|\hat{N}(t)| = \limsup_{t\to\infty} \frac{1}{t}\int_0^t \left(r - \frac{\pi^2}{8L(s)^2} - \frac{1}{2} f'(s)^2 + \frac{L'(s)}{2L(s)}\right) ds\]
and
\[\liminf_{t\to\infty}\frac{1}{t}\log|\hat{N}(t)| = \liminf_{t\to\infty} \frac{1}{t}\int_0^t \left(r - \frac{\pi^2}{8L(s)^2} - \frac{1}{2} f'(s)^2 + \frac{L'(s)}{2L(s)}\right) ds.\]
\end{cor}
This possibility of dramatic oscillation in the number of particles at large times is not usually seen in the branching processes literature. Example \ref{nondiff}, in Section \ref{examples_section} below, helps to show why it occurs in our situation.

\subsection{The critical case, $S=0$}\label{beta_sec}
At least one obvious question immediately arises: what happens when $S=0$? This is an interesting but delicate matter: one must look at the finer behaviour of
\[\int_0^t \left(r - \frac{1}{2} f'(s)^2 - \frac{\pi^2}{8L(s)^2} + \frac{L'(s)}{2L(s)}\right) ds.\]
Our methods, as they stand, are not always sharp enough to say what will happen, and we are unable to provide a complete theory as we must adapt carefully to the set in question. There are several situations, however, where something can be done. We are able to give results on the behaviour near the critical line $\sqrt{2r}t$ in Theorems \ref{betaex} and \ref{thirdex} below. Proofs of these two theorems will be given in Section \ref{critical_sec}, as adaptations of our main proof, that of Theorem \ref{mainthm}.

Fix $\alpha>0$, $\beta\in(0,1)$ and $\gamma>0$, and for $t\geq0$ let
\[f(t)=\alpha + \sqrt{2r}t - \alpha(t+1)^\beta \hs\hbox{ and }\hs L(t) = \gamma(t+1)^\beta.\]

\begin{thm}\label{betaex}
If $\beta<1/3$ then we have $\Pb(\Upsilon=\infty)=0$, and
\[\frac{\log\Pb(\hat N(t)\neq\emptyset)}{t^{1-2\beta}} \longrightarrow -\frac{\pi^2}{8\gamma^2(1-2\beta)}.\]
If $\beta>1/3$, we have $\Pb(\Upsilon=\infty)>0$, and almost surely on survival
\[\frac{\log|\hat N(t)|}{t^\beta} \longrightarrow (\alpha+\gamma)\sqrt{2r}.\]
\end{thm}

It is well-known that the asymptotic speed of the right-most particle in a BBM is $\sqrt{2r}$. The theorem above concerns asking particles to stay close to this critical line forever: for example, we might ask particles to be in $(\sqrt{2r}t - 2\alpha t^\beta, \sqrt{2r}t)$ \emph{for all times} $t\geq0$. If $\beta>1/3$ then particles manage this with positive probability; if $\beta<1/3$ then they do not. What if $\beta=1/3$? Intuitively this question is ``even more critical'' than the previous theorem. Indeed, our methods are not able to give a full answer, but they can identify regimes where each behaviour (growth or death) is observed.

\begin{thm}\label{thirdex}
Consider the case $\beta=1/3$. Let
\[\gamma_0:= \left(\frac{3\pi^2}{8\sqrt{2r}}\right)^{1/3} \hs\hs \hbox{and} \hs\hs \gamma_1:= \left(\frac{3\pi^2}{4\sqrt{2r}}\right)^{1/3}.\]
If $\gamma<\gamma_0$ and $\alpha<\frac{3\pi^2}{8\gamma^2\sqrt{2r}}-\gamma$, then $\Pb(\Upsilon=\infty)=0$; in fact
\[\liminf_{t\to\infty}\frac{\log \Pb(\hat N(t)\neq\emptyset)}{t^{1/3}} \geq \alpha\sqrt{2r} - \frac{3\pi^2}{8\gamma^2} - \gamma\sqrt{2r}\]
and
\[\limsup_{t\to\infty}\frac{\log \Pb(\hat N(t)\neq\emptyset)}{t^{1/3}} \leq \alpha\sqrt{2r} - \frac{3\pi^2}{8\gamma^2} + \gamma\sqrt{2r}.\]
On the other hand, if $\gamma\geq\gamma_1$ and $\alpha>3\gamma_1/2$, or if $\gamma<\gamma_1$ and $\alpha>\gamma+\frac{3\pi^2}{8\gamma^2\sqrt{2r}}$, then $\Pb(\Upsilon=\infty)>0$ and almost surely on survival
\[\liminf_{t\to\infty}\frac{\log |\hat N(t)|}{t^{1/3}} \geq \alpha\sqrt{2r} - \frac{3\pi^2}{8(\gamma\vee\gamma_1)^2} - (\gamma\vee\gamma_1)\sqrt{2r}\]
and
\[\limsup_{t\to\infty}\frac{\log |\hat N(t)|}{t^{1/3}} \leq \alpha\sqrt{2r} - \frac{3\pi^2}{8\gamma^2} + \gamma\sqrt{2r}.\]
\end{thm}

Theorems \ref{betaex} and \ref{thirdex} should be compared with what is currently known about the right-most particle, for example the work of Bramson \cite{bramson:maximal_displacement_BBM} and Lalley and Sellke \cite{lalley_sellke:conditional_limit_frontier_BBM}, results on branching Brownian motion with killing, for example Kesten \cite{kesten:BBM_with_absorption}, and work on the branching random walk, for example Hu and Shi \cite{hu_shi:minimal_pos_crit_mg_conv_BRW} and Jaffuel \cite{jaffuel:crit_barrier_brw_absorption}. The recent article by Jaffuel \cite{jaffuel:crit_barrier_brw_absorption}, in particular, gives results similar to our Theorems \ref{betaex} and \ref{thirdex}.

\section{Examples}\label{examples_section}

We now consider some very simple examples to give the reader a flavour of the implications of Theorem \ref{mainthm}. More complex examples will be given in Sections \ref{extensions} and \ref{critical_sec} in order to explore the limits of our method.

\begin{ex}
Take $f(t) = \lambda t$ with $\lambda\in\Rb$ and $L(t)\equiv L>0$. We have a growth rate of $r - \frac{\lambda^2}{2} - \frac{\pi^2}{8L^2}$ (provided this is non-zero): if this constant is negative, then
\[\frac{1}{t}\log\Pb(\hat N(t)\neq\emptyset) \hs\longrightarrow\hs r - \frac{\lambda^2}{2} - \frac{\pi^2}{8L^2}\]
and if it is positive then there is a positive probability of survival, and almost surely on that event
\[\frac{1}{t}\log\hat N(t) \hs\longrightarrow\hs r - \frac{\lambda^2}{2} - \frac{\pi^2}{8L^2}\]
Thus taking a fixed $L$ introduces an extra ``killing'' rate of $\frac{\pi^2}{8L^2}$ to the system compared to the scaled results of \cite{git:almost_sure_path_properties, hardy_harris:spine_bbm_large_deviations, harris_roberts:scaled_growth, lee:large_deviations_for_branching_diffusions}.
\end{ex}

\begin{ex}
Again take $f(t) = \lambda t$ with $\lambda\in\Rb\setminus\{\sqrt{2r}\}$ but now let $L$ be any unbounded monotone non-decreasing function such that $(f,L)$ satisfies the usual conditions (for example $L(t)=(t+1)^\beta$ with $\beta\in(0,1)$ or $L(t)=\log(t+2)$). Then we have a growth rate of $r-\frac{\lambda^2}{2}$: thus while constant $L$ severely restricts the growth of the system, as soon as we relax $L$ slightly we regain the full growth behaviour seen in \cite{git:almost_sure_path_properties, hardy_harris:spine_bbm_large_deviations, harris_roberts:scaled_growth, lee:large_deviations_for_branching_diffusions}.
\end{ex}

\begin{ex}
Let $f(t) = \sqrt{2r}t$ and $L(t)\equiv L>0$. Then we have extinction almost surely --- and the same applies to any $f$ such that $t^{-1}\int_0^t f'(s)^2 ds \to 2r$ when we take fixed $L$.  We note that Theorems \ref{betaex} and \ref{thirdex} provide much more interesting results in the same area.
\end{ex}

\begin{ex}
Let $f(t) = \lambda(t+1)\sin(\log(t+1))$ and $L(t)\equiv L$. If $r$ is large enough then, on survival, the number of particles alive at time $t$ oscillates, with
\[\liminf_{t\to\infty}\frac{1}{t}\log|\hat{N}(t)| = r - \frac{\pi^2}{8L^2} - \frac{\lambda^2}{\sqrt5}\left(\frac{\sqrt5+1}{2}\right)\]
and
\[\limsup_{t\to\infty}\frac{1}{t}\log|\hat{N}(t)| = r - \frac{\pi^2}{8L^2} - \frac{\lambda^2}{\sqrt5}\left(\frac{\sqrt5-1}{2}\right).\]
(Note the appearance of the golden ratio.)
\end{ex}

The reason for this oscillation on the exponential scale becomes clearer when we consider the following simpler, but perhaps less natural, example.

\begin{ex}\label{nondiff}
Define a continuous function $f:[0,\infty)\to\Rb$ by setting $f(t)=0$ for $t\in[0,1]$ and
\[f'(t) = \left\{ \begin{array}{ll} 0 & \hbox{ if } \hsl 2^{2k} \leq t < 2^{2k+1} \hsl \hbox{ for some } \hsl k\in\{0,1,2,\ldots\}\\
									1 & \hbox{ if } \hsl 2^{2k+1} \leq t < 2^{2k+2} \hsl \hbox{ for some } \hsl k\in\{0,1,2,\ldots\}\end{array}\right. .\]
Then, provided that $r > \frac{1}{3} + \frac{\pi^2}{8L^2}$, on non-extinction we have
\[\liminf_{t\to\infty}\frac{1}{t}\log|\hat{N}(t)| = r - \frac{\pi^2}{8L^2} - \frac{1}{3}\]
and
\[\limsup_{t\to\infty}\frac{1}{t}\log|\hat{N}(t)| = r - \frac{\pi^2}{8L^2} - \frac{1}{6}.\]
The idea here is that the number of particles grows quickly when $f'(t)=0$, but much more slowly when $f'(t)=1$ as the steep gradient means that particles have to struggle to follow the path for a long time. As the size of the intervals $[2^n, 2^{n+1}]$ grows exponentially, the behaviour of the number of particles at time $t$ is dominated by the behaviour on the most recent such interval. [We note that this choice of $f$ is not twice differentiable; however, it can be uniformly approximated by twice differentiable functions, and it is easily checked that our results still hold - see Section \ref{extensions}.]
\end{ex}

\section{The spine setup}\label{spine_section}

\noindent
Consider a dyadic one-dimensional branching Brownian motion, branching at rate $r$, with associated probability measure $\Pb$ under which
\begin{itemize}
\item{we begin with a root particle, $\emptyset$, at 0;}
\item{if a particle $u$ is in the tree then all its ancestors are also in the tree (if $v$ is an ancestor of $u$ then we write $v < u$);}
\item{each particle $u$ has a lifetime $\sigma_u$, which is exponentially distributed with parameter $r$, and a fission time $S_u = \sum_{v\leq u}\sigma_v$;}
\item{each particle $u$ has a position $X_u(t) \in \Rb$ at each time $t\in [S_u-\sigma_u, S_u)$;}
\item{at the fission time $S_u$, $u$ has disappeared and been replaced by two children $u0$ and $u1$, which inherit the position of their parent;}
\item{given its birth time and position, each particle $u$, while alive, moves according to a standard Brownian motion started from $X_u(S_u-\sigma_u)$ independently of all other particles.}
\end{itemize}
For convenience, we extend the position of a particle $u$ to all times $t\in[0, S_u)$, to include the paths of all its ancestors:
\[X_u(t):= X_v(t) \hbox{ if } v\leq u \hbox{ and } S_v - \sigma_v  \leq t < S_v.\]
We recall that we defined $N(t)$ to be the set of particles alive at time $t$,
\[N(t):=\{u: S_u - \sigma_u \leq t < S_u\},\]
and also that
\[\hat N(t) := \left\{u \in N(t) : |X_u(s)-f(s)|< L(s) \hs \forall s\leq t\right\}.\]

We choose from our BBM one distinguished line of descent or \emph{spine} -- that is, a subset $\xi$ of the tree such that $\xi\cap N(t)$ contains exactly one particle for each $t$ and if $u\in \xi$ and $v<u$ then $v\in \xi$. We make this choice as follows:
\begin{itemize}
\item{the initial particle $\emptyset$ is in the spine;}
\item{at the fission time of node $u$ in the spine, the new spine particle is chosen uniformly at random from the two children $u0$ and $u1$ of $u$.}
\end{itemize}
We denote the position of the spine particle at time $t$ by $\xi_t$; however we may also occasionally use $\xi_t$ to refer to the spine particle itself (that is, the node of the tree that is in the spine at time $t$) --- it should be clear from the context which meaning is intended. We call the resulting probability measure (on the space of \emph{marked trees with spines}) $\Pt$. We also consider the translated probability measures $\Pb_x$ and $\Pt_x$ for $x\in\Rb$, where under $\Pb_x$ and $\Pt_x$ we start with a single particle at $x$ instead of 0.

\subsection{Filtrations}
We use three different filtrations, $\Fg_t$, $\Ft_t$ and $\Gg_t$, to encapsulate different amounts of information. We give descriptions of these filtrations here, but the reader is referred to Hardy and Harris \cite{hardy_harris:spine_approach_applications} for the full definitions.

\begin{itemize}
\item{$\Fg_t$ contains all the information about the marked tree up to time $t$. However, it does not know which particle is the spine at any point.}
\item{$\Ft_t$ contains all the information about both the marked tree and the spine up to time $t$.}
\item{$\Gg_t$ contains just the spatial information about the spine up to time $t$; it does not know anything about the rest of the tree.}
\end{itemize}
We note that $\Fg_t \subseteq \Ft_t$ and $\Gg_t \subseteq \Ft_t$, and also that $\Pt_x$ is an extension of $\Pb_x$ in that $\Pb_x = \Pt_x |_{\Fg_\infty}$.

\subsection{Martingales and a change of measure}\label{unscaled_measure_change}

Under $\Pt$, the path of the spine $(\xi_t,\hsl t\geq 0)$ is a standard Brownian motion. Set
\begin{multline*}
G(t) := \exp\left( \int_0^t f'(s) d\xi_s - \frac{1}{2}\int_0^t f'(s)^2 ds + \int_0^t \frac{\pi^2}{8L(s)^2}ds \right)\\
\cdot\exp\left(\frac{L'(t)}{2L(t)}(\xi_t-f(t))^2 - \int_0^t \left(\frac{L''(s)}{2L(s)}(\xi_s-f(s))^2 + \frac{L'(s)}{2L(s)}\right)ds \right).
\end{multline*}
We claim that the process
\[V(t) := G(t) \cos\left(\frac{\pi}{2L(t)}(\xi_t-f(t))\right), \hs t\geq0\]
is a $\Gg_t$-local martingale.

\begin{lem}\label{elocalmg}
Let
\[F(t):= \exp\left(\int_0^t \frac{\pi^2}{8L(s)^2}ds + \frac{L'(t)}{2L(t)}\xi_t^2 - \int_0^t \left(\frac{L''(s)}{2L(s)}\xi_s^2 + \frac{L'(s)}{2L(s)}\right)ds\right).\]
The process
\[U(t) := F(t) \cos\left(\frac{\pi \xi_t}{2L(t)}\right)\]
is a $\Gg_t$-local martingale.
\end{lem}

\begin{proof}
By It\^o's formula,
\begin{align*}
dU(t) &= \frac{\pi^2}{8L(t)^2} F(t) \cos\left(\frac{\pi \xi_t}{2L(t)}\right) dt\\
& \hspace{5mm} + \left(\frac{L''(t)}{2L(t)} - \frac{L'(t)^2}{2L(t)^2}\right)\xi_t^2 F(t) \cos\left(\frac{\pi \xi_t}{2L(t)}\right) dt\\
& \hspace{5mm} - \left(\frac{L''(t)}{2L(t)}\xi_t^2 + \frac{L'(t)}{2L(t)}\right) F(t) \cos\left(\frac{\pi \xi_t}{2L(t)}\right) dt\\
& \hspace{5mm} + \frac{\pi L'(t)}{2L(t)^2} \xi_t F(t) \sin\left(\frac{\pi \xi_t}{2L(t)}\right) dt \\
& \hspace{5mm} + \frac{L'(t)}{L(t)}\xi_t F(t) \cos\left(\frac{\pi \xi_t}{2L(t)}\right) d\xi_t\\
& \hspace{5mm} + \frac{\pi}{2L(t)} F(t) \sin\left(\frac{\pi \xi_t}{2L(t)}\right) d\xi_t\\
& \hspace{5mm} + \left(\frac{L'(t)}{2L(t)} + \frac{L'(t)^2}{2L(t)^2}\xi_t^2\right) F(t) \cos\left(\frac{\pi \xi_t}{2L(t)}\right)dt\\
& \hspace{5mm} - \frac{\pi^2}{8L(t)^2} F(t) \cos\left(\frac{\pi \xi_t}{2L(t)}\right)dt\\
& \hspace{5mm} - \frac{\pi L'(t)}{2L(t)^2} \xi_t F(t) \sin\left(\frac{\pi \xi_t}{2L(t)}\right) dt.\qedhere
\end{align*}
\end{proof}

\begin{lem}
The process $V(t)$, $t\geq0$ is a $\Gg_t$-local martingale.
\end{lem}

\begin{proof}
Again applying It\^o's formula does the trick - or one may simply apply Girsanov's theorem in series with Lemma \ref{elocalmg}.
\end{proof}

By stopping the process $V(t)$ at the first exit time of the spine particle from the tube $\{(x,t): |f(t)-x|< L(t)\}$, we obtain also that
\[\zeta(t) := V(t) \ind_{\{|f(s)-\xi_s| < L(s) \hsl \forall s\leq t\}}\]
is a $\Gg_t$-local martingale, and in fact since its size is constrained it is easily seen to be a $\Gg_t$-martingale. We call this martingale $\zeta$ the \emph{single-particle martingale}.

\begin{defn}
We define an $\Ft_t$-adapted martingale by
\[\tilde{\zeta}(t) = 2^{n(\xi,t)} \times e^{-rt} \times \zeta(t),\]
where $n(\xi,t):=|\{v: v<\xi_t\}|$ is the generation of the spine at time $t$. The proof that this process is an $\Ft_t$-martingale can be found in \cite{hardy_harris:spine_approach_applications}.

We note that if $f$ is an $\Ft_t$-measurable function then we can write:
\begin{equation}
f(t)=\sum_{u\in N_t}f_u(t) \ind_{\xi_t=u} \label{unscaled_fdecomp}
\end{equation}
where each $f_u$ is $\Fg_t$-measurable -- intuitively, if $f$ is in fact $\Gg_t$-measurable, one replaces every appearance of $\xi_t$ with $X_u(t)$: so for example
\begin{multline*}
G_u(t) := \exp\left( \int_0^t f'(s) dX_u(s) - \frac{1}{2}\int_0^t f'(s)^2 ds + \int_0^t \frac{\pi^2}{8L(s)^2}ds \right)\\
\cdot\exp\left(\frac{L'(t)}{2L(t)}(X_u(t)-f(t))^2 - \int_0^t \left(\frac{L''(s)}{2L(s)}(X_u(s)-f(s))^2 + \frac{L'(s)}{2L(s)}\right)ds \right).
\end{multline*}
It is also shown in \cite{hardy_harris:spine_approach_applications} that if we define
\[Z(t) := \sum_{u\in N(t)} e^{-rt}\zeta_u(t),\]
where $\zeta_u$ is the $\Fg_t$-adapted process defined via the representation of $\zeta$ as in (\ref{unscaled_fdecomp}), then
\[Z(t) = \Pt[\tilde \zeta(t) | \Fg_t]\]
and hence that $Z$ is an $\Fg_t$-martingale. This martingale is the main object of interest in this article.
\end{defn}

\begin{defn}
We define a new measure, $\Qt_x$, via
\[\left.\frac{d\Qt_x}{d\Pt_x}\right|_{\Ft_t} = \frac{\tilde{\zeta}(t)}{\tilde{\zeta}(0)}.\]
Also, for convenience, define $\Qb_x$ to be the projection of the measure $\Qt$ onto $\Fg_\infty$; then
\[\left.\frac{d\Qb_x}{d\Pb_x}\right|_{\Fg_t} = \frac{Z(t)}{Z(0)}.\]
\end{defn}

\begin{lem}
Under $\Qt_x$,
\begin{itemize}
\item when at position $y$ at time $t$ the spine $\xi$ moves as a Brownian motion with drift
\[f'(t) + (y-f(t))\frac{L'(t)}{L(t)} - \frac{\pi}{2L(t)} \tan\left(\frac{\pi}{2L(t)}(y - f(t))\right);\]
\item the fission times along the spine occur at an accelerated rate $2r$;
\item at the fission time of node $v$ on the spine, the single spine particle is replaced by two children, and the new spine particle is chosen uniformly from the two children;
\item the remaining child gives rise to an independent subtree, which is not part of the spine and which (along with its descendants) draws out a marked tree determined by an independent copy of the original measure $\Pb$ shifted to its position and time of birth.
\end{itemize}
\end{lem}

\noindent
This, again, was covered in \cite{hardy_harris:spine_approach_applications}. We also use that, under $\Qt_x$, the spine remains within distance $L(t)$ of $f(t)$ for all times $t\geq0$. To see this explicitly, note that
\[\Qt_x(\xi_t\not\in\hat{N}(t)) = \Pt_x\left[\ind_{\{\xi_t\not\in\hat{N}(t)\}}\frac{\tilde\zeta(t)}{\tilde\zeta(0)}\right] = 0\]
by definition of $\tilde\zeta(t)$. All other particles, once born, move like independent standard Brownian motions but -- as under $\Pb_x$ -- we imagine them being ``killed'' instantly upon leaving the tube of radius $L$ about $f$. In reality they are still present in the system, but make no contribution to $Z$ once they have left the tube.

\begin{rmk}
Note that $\hat{N}$, and hence $Z$, $\Qt$ and various other of our constructions, depend upon the choice of function $f$ and radius $L$. Usually these will be implicit, but occasionally we shall write $\hat{N}^{f,L}$, $Z^{f,L}$ and $\Qt^{f,L}$ (and so on) to emphasise the choice of $f$ and $L$ in use at the time.
\end{rmk}

\subsection{Spine tools}

We now state the spine decomposition theorem, which will be a vital tool in our investigation. It allows us to relate the growth of the whole process to just the behaviour along the spine. For a proof (of a more general version) the reader is again referred to \cite{hardy_harris:spine_approach_applications}.

\begin{thm}[Spine decomposition]
We have the following decomposition of $Z$:
\[\Qt_x[Z(t)|\Gg_\infty] = \int_0^t 2r e^{-rs}\zeta(s) ds + e^{-rt}\zeta(t).\]
\end{thm}

The spine decomposition is usually used in conjunction with a result like the following -- a proof of a more general form of this lemma can be found in \cite{lyons_peres:probability_on_trees}.

\begin{lem}
\label{stdmeas}
Let $Z(\infty)=\limsup_{t\to\infty} Z(t)$. Then
\[\Qb \ll \Pb \hs \Leftrightarrow \hs Z(\infty)<\infty \hs \Qb\hbox{-a.s. } \Leftrightarrow \hs \Qb = Z(\infty) \Pb\]
and
\[\Qb \perp \Pb \hs \Leftrightarrow \hs Z(\infty)=\infty \hs \Qb\hbox{-a.s. } \Leftrightarrow \hs \Pb[Z(\infty)]=0.\]
\end{lem}

Another extremely useful spine tool is the \emph{many-to-one} theorem. A much more general version of this theorem is proved in \cite{hardy_harris:spine_approach_applications}, but the following version will be enough for our purposes.

\begin{thm}[Many-to-One]
\label{many_to_one}
If $f(t)$ is $\Gg_t$-measurable for each $t\geq0$ with representation \emph{(\ref{unscaled_fdecomp})}, then
\[\Pb\left[\sum_{u\in N(t)} f_u(t)\right] = e^{rt}\Pt[f(t)].\]
\end{thm}

We have one more lemma, a proof of which can be found in \cite{harris_roberts:extinction_letter}. Although this result is extremely simple --- and essential to our study --- we are not aware of its presence in the literature before \cite{harris_roberts:extinction_letter}.

\begin{lem}
\label{extinction_lem}
For any $t\in[0,\infty]$ (note that infinity is included here), we have
\[\Pb_x(Z(t)>0) = \Qb_x\left[\frac{Z(0)}{Z(t)}\right].\]
\end{lem}

\section{Almost sure growth along paths}

\subsection{Controlling the measure change}
Before applying the tools that we have developed, we need the following short lemma to keep the Girsanov part of our change of measure under control.

\begin{lem}
\label{unscaled_int_lem}
For any $u\in\hat{N}(t)$, almost surely under both $\Pt_x$ and $\Qt_x$ we have
\[\left|\int_0^t f'(s) dX_u(s) - \int_0^t f'(s)^2 ds\right| \leq |f'(t)|L(t) + |f'(0)|x + \int_0^t |f''(s)|L(s) ds\]
and hence under $\Pt$
\begin{multline}\label{gineq}
\exp\left(\frac{1}{2}\int_0^t f'(s)^2 ds + \int_0^t \frac{\pi^2}{8L(s)^2}ds - \int_0^t \frac{L'(s)}{2L(s)}ds - E(t)\right)\\
\leq G_u(t)\leq \exp\left(\frac{1}{2}\int_0^t f'(s)^2 ds + \int_0^t \frac{\pi^2}{8L(s)^2}ds - \int_0^t \frac{L'(s)}{2L(s)}ds + E(t)\right).
\end{multline}
\end{lem}

\begin{proof}
From the integration by parts formula for It\^o calculus, we know that
\[f'(t)X_u(t) = f'(0)X_u(0) + \int_0^t f''(s)X_u(s) ds + \int_0^t f'(s)dX_u(s).\]
From ordinary integration by parts,
\[\int_0^t f'(s)^2 ds = f'(t)f(t) - f'(0)f(0) - \int_0^t f(s)f''(s) ds.\]
We also note that if $u\in\hat N(t)$ then $|X_u(s) - f(s)|< L(s)$ for all $s\leq t$. Thus
\begin{align*}
\biggl|&\int_0^t f'(s) dX_u(s) - \int_0^t f'(s)^2 ds\hspace{0.5mm}\biggr|\\
&= \biggl|\hspace{0.5mm}f'(t)(X_u(t)-f(t)) - f'(0)(X_u(0)-f(0)) - \int_0^t f''(s)(X_u(s)-f(s))ds\hspace{0.5mm}\biggr|\\
&\leq |f'(t)|L(t) + |f'(0)|x + \int_0^t |f''(s)|L(s)ds.
\end{align*}
Plugging this estimate into the definition of $G_u(t)$ gives the result.
\end{proof}

We are now ready to prove our first real result.

\begin{prop}\label{uiprop}
Recall that $Z(\infty):= \limsup_{t\to\infty}Z(t)$. If $S < 0$, then the process almost surely becomes extinct in finite time (and hence we have $Z(\infty)=0$). In this case,
\[\frac{\log \Pb(\hat N(t) \neq \emptyset)}{\inf_{s\leq t}\int_0^s (r - \frac{\pi^2}{8L(u)^2}-\frac{1}{2}f'(u)^2 + \frac{L'(u)}{2L(u)})du} \hs\longrightarrow\hs 1.\]
Alternatively, if $S > 0$ then $\Pb[Z(\infty)]=1$.
\end{prop}

\begin{proof}
We first recall the spine decomposition and apply inequality (\ref{gineq}):
\begin{align*}
\Qt[Z(t)|\Gg_\infty] &= \int_0^t 2r e^{-rs}\zeta(s) ds + e^{-rt}\zeta(t)\\
&\leq \int_0^t 2r e^{-\int_0^s(r-\frac{\pi^2}{8L(u)^2} - \frac{1}{2}f'(u)^2 + \frac{L'(u)}{2L(u)})du + E(s)} ds\\
&\hspace{25mm} + e^{-\int_0^t(r-\frac{\pi^2}{8L(u)^2} - \frac{1}{2}f'(u)^2 + \frac{L'(u)}{2L(u)})du + E(t)}.
\end{align*}
If $S>0$, then the integrand above is exponentially small for all large $t$ (as is the second term); so $\liminf_{t\to\infty}\Qt[Z(t)|\Gg_\infty]<\infty$. It is easy to show that $1/Z$ is a positive $(\Qt,\Fg_t)$-supermartingale, and hence $Z(t)$ converges $\Qt$-almost surely to some (possibly infinite) limit. Thus, applying Fatou's lemma, we get
\[\Qt[Z(\infty)|\Gg_\infty]\leq \liminf_{t\to\infty} \Qt[Z(t)|\Gg_\infty] <\infty.\]
We deduce that $Z(\infty)<\infty$ $\Qt$-almost surely, and Lemma \ref{stdmeas} then gives that $\Pb[Z(\infty)]=1$.

Alternatively, suppose that $S<0$. Then by the above,
\[\Qt[Z(t)|\Gg_\infty] \leq (2rt+1)e^{-\inf_{s\leq t}\left\{\int_0^s(r-\frac{\pi^2}{8L(u)^2} - \frac{1}{2}f'(u)^2 + \frac{L'(u)}{2L(u)})du - E(s)\right\}}.\]
Now, by the tower property of conditional expectation and Jensen's inequality,
\[\Pb(\hat N(t)\neq\emptyset) = \Pb(Z(t)>0) = \Qb\left[\frac{1}{Z(t)}\right] \geq \Qt\left[\frac{1}{\Qt[Z(t)|\Gg_\infty]}\right].\]
This clearly implies that, for large $t$ (using that $S<0$),
\begin{multline*}
\frac{\log\Pb(\hat N(t)\neq\emptyset)}{\inf_{s\leq t}\int_0^s (r - \frac{\pi^2}{8L(u)^2} - \frac{1}{2}f'(u)^2 + \frac{L'(u)}{2L(u)})du}\\
\leq \frac{\inf_{s\leq t}\left\{\int_0^s (r - \frac{\pi^2}{8L(u)^2} - \frac{1}{2}f'(u)^2 + \frac{L'(u)}{2L(u)})du - E(s) \right\} - \log(2rt+1)}{\inf_{s\leq t}\int_0^s (r - \frac{\pi^2}{8L(u)^2} - \frac{1}{2}f'(u)^2 + \frac{L'(u)}{2L(u)})du};
\end{multline*}
and it is easy to see that the right-hand side converges to one as $t\to\infty$. This gives us our upper bound.

For the lower bound (still in the case $S<0$), suppose for a moment that we may choose $\gamma>1$ such that
\[\liminf_{t\to\infty} \frac{1}{t}\int_0^t \left(r - \frac{1}{2} f'(s)^2 - \frac{\pi^2}{8\gamma L(s)^2} + \frac{L'(s)}{2L(s)}\right) ds < 0.\]
We note that we may choose $\gamma$ in this way if $\int_0^t \pi^2/8L(s)^2 ds$ (eventually) shows at most linear growth, which we will check later.
Then
\begin{align*}
\Pb(\hat N(t)\neq\emptyset) \leq \inf_{s\leq t}\Pb(\hat N(s)\neq\emptyset) &= \inf_{s\leq t}\Pb\left[\frac{Z^{f,\gamma L}(s)}{Z^{f,\gamma L}(s)}\ind_{\{\hat N^{f,L}(s)\neq\emptyset\}}\right]\\
&= \inf_{s\leq t}\Qb^{f,\gamma L}\left[\frac{1}{Z^{f,\gamma L}(s)}\ind_{\{\hat N^{f,L}(s)\neq\emptyset\}}\right]\\
&\leq \inf_{s\leq t}\Qb^{f,\gamma L}\left[\frac{\ind_{\{\hat N^{f,L}(s)\neq\emptyset\}}}{\sum_{v\in\hat N^{f,L}(s)}e^{-rs}\zeta^{f,\gamma L}_v(s)}\right].
\end{align*}
If $\hat{N}^{f,L}(s)\neq\emptyset$ then there is at least one particle $v$ in $\hat{N}^{f,L}(s)$; we may then apply inequality (\ref{gineq}) to $\zeta_v^{f,\gamma L}(s)$ see that
\[\Pb(\hat N(t)\neq\emptyset) \leq \inf_{s\leq t}\frac{1}{e^{-\int_0^s (r - \frac{\pi^2}{8\gamma^2 L(u)^2} - \frac{1}{2}f'(u)^2 + \frac{L'(u)}{2L(u)})du - \gamma^2 E(s)}\cos\left(\pi/2\gamma\right)}.\]
We repeat our calculations from the upper bound, taking logarithms and dividing by the desired denominator, to give
\begin{align}
&\frac{\log\Pb(\hat N(t)\neq\emptyset)}{\inf_{s\leq t}\int_0^s (r - \frac{\pi^2}{8L(u)^2} - \frac{1}{2}f'(u)^2 + \frac{L'(u)}{2L(u)})du}\notag\\
&\geq \frac{\inf_{s\leq t}\left\{\int_0^s (r - \frac{\pi^2}{8\gamma^2 L(u)^2} - \frac{1}{2}f'(u)^2 + \frac{L'(u)}{2L(u)})du - \gamma^2 E(s) \right\} - \log\cos\left(\pi/2\gamma\right)}{\inf_{s\leq t}\int_0^s (r - \frac{\pi^2}{8L(u)^2} - \frac{1}{2}f'(u)^2 + \frac{L'(u)}{2L(u)})du}\notag\\
&\geq 1 + \frac{\left(1-\frac{1}{\gamma^2}\right)\sup_{s\leq t}\int_0^s \frac{\pi^2}{8L(u)^2}du + \gamma^2\sup_{s\leq t} E(s) - \log\cos\left(\pi/2\gamma\right)}{\inf_{s\leq t}\int_0^s (r - \frac{\pi^2}{8L(u)^2} - \frac{1}{2}f'(u)^2 + \frac{L'(u)}{2L(u)})du}\label{bottomline}
\end{align}
for large $t$. Thus it remains to check that the right-hand side above has a limsup that is close to 1 when $\gamma$ is close to 1. Again it is sufficient that $\int_0^t \pi^2/8L(s)^2 ds$ can (eventually) show at most linear growth, and we check that fact now. This is rather fiddly and not interesting in the context of the rest of the proof. Suppose it is not true; that is, suppose
\[\limsup_{t\to\infty} \frac{1}{t}\int_0^t \frac{\pi^2}{8L(s)^2} ds = \infty.\]
Then since $S>-\infty$ we must have
\begin{equation}\label{templabel}
\limsup_{t\to\infty} \frac{1}{t}\int_0^t \left(\frac{\pi^2}{8L(s)^2} - \frac{L'(s)}{2L(s)}\right) ds < \infty.
\end{equation}
If we take $T_n:=\inf\{t>0 : \int_0^t \pi^2/8L(s)^2 ds > nt\}$, then
\[\frac{d}{dt}\left.\left(\frac{1}{t}\int_0^t \frac{\pi^2}{8L(s)^2}ds\right)\right|_{T_n} > 0,\]
so differentiating and rearranging we get
\[L(T_n)^2 < \frac{\pi^2 T_n}{8\int_0^{T_n} \frac{\pi^2}{8L(s)^2} ds} < \frac{\pi^2 T_n}{8n}.\]
Now, we note that $\int_0^t \frac{L'(s)}{L(s)} ds = \log L(t) - \log L(0)$, so (\ref{templabel}) implies that for all large $t$,
\[\int_0^t \frac{\pi^2}{8L(s)^2}ds < Kt + \frac{1}{2}\log L(t)\]
for some constant $K$. We have just shown that $L(T_n)^2 < \pi^2 T_n/8n$, so for all large $n$,
\[\int_0^{T_n} \frac{\pi^2}{8L(s)^2}ds < KT_n + \frac{1}{4}\log T_n + \frac{1}{4}\log\frac{\pi^2}{8n}\]
contradicting (for large $n$) the definition of $T_n$.

We have shown that
\[\limsup_{t\to\infty} \frac{1}{t}\int_0^t \frac{\pi^2}{8L(s)^2} ds < \infty;\]
which allows us to make the limsup of (\ref{bottomline}) as close to $1$ as we like by letting $\gamma\downarrow1$. This completes the lower bound, which in particular implies (by monotonicity) that the probability of eventual extinction is equal to 1.
\end{proof}

\subsection{Almost sure growth}

Having established, in Proposition \ref{uiprop}, the large deviations behaviour of our model, we now turn to the question of what happens when extinction does not occur. The two propositions in this section contain the meat of our results in this direction. Proposition \ref{liminf_prop} gives a lower bound on the number of particles in $\hat{N}(t)$ for large $t$, and Proposition \ref{limsup_prop} an upper bound. The former holds only on the event that $Z$ has a positive limit; as mentioned in the introduction, this set coincides (up to a null event) with the event that no particle manages to follow within $L$ of $f$, although we will not prove this fact until Section \ref{inv_spine}. The proofs of our two propositions are very simple, but we stress again that this is due to the careful choice of martingale.

\begin{prop}\label{liminf_prop}
Let $\Omega^{\star}$ be the set on which $Z$ has a strictly positive limit,
\[\Omega^{\star} := \left\{\liminf_{t\to\infty} Z(t)>0 \right\}.\]
If $S>0$ then $\Pb$-almost surely on $\Omega^{\star}$ we have
\[\liminf_{t\to\infty} \frac{\log|\hat N(t)|}{\int_0^t \left(r - \frac{1}{2} f'(s)^2 - \frac{\pi^2}{8L(s)^2} + \frac{L'(s)}{2L(s)}\right) ds} \hs\geq\hs 1.\]
\end{prop}

\begin{proof}
For any $t\geq 0$, by inequality (\ref{gineq}), almost surely under $\Pb$
\[Z(t) = \sum_{u\in \hat N(t)} e^{-rt}\zeta_u(t) \leq |\hat N(t)| e^{-\int_0^t (r - \frac{\pi^2}{8 L(s)^2} - \frac{1}{2}f'(s)^2 + \frac{L'(s)}{2L(s)})ds + E(t)}.\]
Hence (for large $t$, since $S>0$)
\begin{multline*}
\frac{\log|\hat N(t)|}{\int_0^t \left(r - \frac{1}{2} f'(s)^2 - \frac{\pi^2}{8L(s)^2} + \frac{L'(s)}{2L(s)}\right) ds}\\
\geq \frac{\log Z(t) + \int_0^t \left(r - \frac{1}{2} f'(s)^2 - \frac{\pi^2}{8L(s)^2} + \frac{L'(s)}{2L(s)}\right)ds - E(t)}{\int_0^t \left(r - \frac{1}{2} f'(s)^2 - \frac{\pi^2}{8L(s)^2} + \frac{L'(s)}{2L(s)}\right) ds}.
\end{multline*}
Now, on $\Omega^{\star}$ we have $\liminf_{t\to\infty} Z(t) >0$ and thus $\frac{1}{\delta t}\log Z(t)$ has a non-negative liminf for any $\delta>0$; then since $S>0$ we see that the right-hand side above has liminf at least 1.
\end{proof}

\begin{rmk}
Recall that under $\Pb$, $Z$ is a non-negative martingale, and hence \mbox{$\liminf_{t\to\infty}Z(t) = Z(\infty)$} $\Pb$-almost surely. If $S > 0$, then by Proposition \ref{uiprop} $\Pb[Z(\infty)]=1$, so in this case $\Omega^{\star}$ occurs with strictly positive probability.
\end{rmk}

\begin{prop}\label{limsup_prop}
If $S>0$, then $\Pb$-almost surely we have
\[\limsup_{t\to\infty} \frac{\log|\hat N(t)|}{\int_0^t \left(r - \frac{1}{2} f'(s)^2 - \frac{\pi^2}{8L(s)^2} + \frac{L'(s)}{2L(s)}\right) ds} \hs\leq\hs 1.\]
\end{prop}

\begin{proof}
Fix $\gamma>1$ and let $\varepsilon=\cos(\pi/2\gamma)$. Since $Z^{f,\gamma L}$ is a non-negative martingale under $\Pb$, we have $Z^{f,\gamma L}(\infty) < \infty$ $\Pb$-almost surely. This implies that for any $\delta>0$, almost surely
\[\limsup_{t\to\infty}\frac{1}{\delta t}\log Z^{f,\gamma L}(t) \leq 0.\]
Now, almost surely under $\Pb$,
\[Z^{f,\gamma L}(t) \hs = \sum_{u\in \hat N^{f,\gamma L}(t)} e^{-rt} \zeta^{f,\gamma L}_u(t) \hs \geq \sum_{u\in \hat N^{f,L}(t)} e^{-rt} \zeta^{f,\gamma L}_u(t).\]
By the definition of $\varepsilon$ above, for any $u\in \hat N^{f,L}(t)$ the cosine term in $\zeta^{f,\gamma L}_u(t)$ is at least $\varepsilon$ (since the particle is within $L$ of $f(t)$ at time $t$). Applying inequality (\ref{gineq}) we see that
\[Z^{f,\gamma L}(t) \geq |\hat N^{f,L}| \cdot \varepsilon \cdot e^{-\int_0^t (r - \frac{\pi^2}{8\gamma^2 L(s)^2} - \frac{1}{2}f'(s)^2 + \frac{L'(s)}{2L(s)})ds - \gamma^2 E(t)}\]
and hence
\begin{multline*}
\frac{\log|\hat N(t)|}{\int_0^t \left(r - \frac{1}{2} f'(s)^2 - \frac{\pi^2}{8L(s)^2} + \frac{L'(s)}{2L(s)}\right) ds}\\
\leq \frac{\log Z(t) - \log\varepsilon + \int_0^t \left(r - \frac{1}{2} f'(s)^2 - \frac{\pi^2}{8\gamma^2 L(s)^2} + \frac{L'(s)}{2L(s)}\right)ds + \gamma^2 E(t)}{\int_0^t \left(r - \frac{1}{2} f'(s)^2 - \frac{\pi^2}{8L(s)^2} + \frac{L'(s)}{2L(s)}\right) ds}.
\end{multline*}
As in Proposition \ref{uiprop}, we can bound the growth of the $\int_0^t \frac{\pi^2}{8\gamma^2 L(s)^2} ds$ term in the numerator so that letting $\gamma\downarrow 1$ we get the desired result.
\end{proof}

\begin{cor}\label{lim_cor}
If $S > 0$, then $\Pb$-almost surely on the event $\Omega^{\star}$,
\[\frac{\log |\hat N(t)|}{\int_0^t (r - \frac{\pi^2 t}{8L^2} - \frac{1}{2}\int_0^t f'(s)^2 + \frac{L'(s)}{2L(s)}) ds} \hs\longrightarrow\hs 1.\]
\end{cor}

\begin{proof}
Simply combine Propositions \ref{liminf_prop} and \ref{limsup_prop}.
\end{proof}

\section{Showing that $Z(\infty)=0$ agrees with extinction}\label{inv_spine}
We note that we have now established our main result except for one key point: our growth results have so far been on the event $\{Z(\infty)>0\}$, rather than the event of survival of the process, $\{\Upsilon = \infty\}$. We turn now to showing that these two events differ only on a set of zero probability.

The approach to proving this is often analytic: one shows that \mbox{$\Pb(Z(\infty)>0)$} and \mbox{$\Pb(\Upsilon=\infty)$} satisfy the same differential equation with the same boundary conditions, and then shows that any such solution to the equation is unique. There is also sometimes a probabilistic approach to such arguments: one considers the product martingale
\[P(t):=\Pb(Z(\infty)=0 | \Fg_t) = \prod_{u\in N(t)} \Pb_{X_u(t)}(Z(\infty)=0).\]
On extinction, the limit of this process is clearly 1, and if we could show that on survival the limit is 0, then since $P$ is a bounded non-negative martingale we would have
\[\Pb(\Upsilon<\infty) = \Pb[P(\infty)] = \Pb[P(0)] = \Pb(Z(\infty)=0).\]
In Harris et al \cite{harris_et_al:fkpp_one_sided_travelling_waves}, for example, we have killing of particles at the origin rather than on the boundary of a tube -- and it is shown that on survival, at least one particle escapes to infinity and its term in the product martingale tends to zero. This is enough to complete the argument (although in \cite{harris_et_al:fkpp_one_sided_travelling_waves} the authors favour the analytic approach). In our case we are hampered by the fact that for a single particle $u$ the value of $\Pb_{X_u(t)}(Z_u(\infty)=0)$ is bounded away from zero, and if the particle is close to the edge of the tube, or even possibly in some places in the interior the tube, then this probability takes values arbitrarily close to 1.

The time-inhomogeneity of our problem means that other standard methods also fail. Our alternative approach is based upon similar principles as the probabilistic approach above, but is more direct: we show that if at least one particle survives for a long time, then it will have many births in ``good'' areas of the tube, and thus $Z(\infty)>0$ with high probability.

Recall that under $\Pt_x$, we start at time $t=0$ with one particle at position $x$ (rather than at the origin) -- and similarly for $\Qt_x$. We assume throughout this section that $S>0$, otherwise there is nothing to prove --- our theorem does not consider the case $S=0$, and if $S<0$ we have proved that $\Pb(\Upsilon=\infty)=0=\Pb(Z(\infty)=0)$. We now need some more notation.

\begin{defn}
Let $L_0 := \frac{\pi}{2\sqrt S}\vee 1$, and define
\[\begin{array}{lccl} \tilde L: & [0,\infty) &\to     &(0,\infty)\\
						   & t     &\mapsto  &\left\{\begin{array}{ll}L(t) & \hbox{ if } L(t)\leq L_0 \\ L_0 + (L(t)-L_0)e^{-(L(t)-L_0)^2} & \hbox{ if } L(t)>L_0\end{array}\right.\end{array}\]
and
\[\begin{array}{lccl} \tilde f: & [0,\infty) &\to     &\Rb\\
						   & t     &\mapsto  &f(t) + L(t) - \tilde L(t).\end{array}\]
Now, for any function $g$ on $[0,\infty)$, define the $t$-delayed version $g_t$ of $g$ for $t\in[0,\infty)$ by
\[g_t(s) = g(t+s)-g(t), \hs s\geq0.\]
Thus for each $t\geq0$ we have four new functions $f_t$, $\tilde f_t$, $L_t$ and $\tilde L_t$.

Also, for $\alpha\in[0,1)$, define
\[U_\alpha = \{(t,x) : \Pb_{x-f(t)}(Z^{f_t, L_t}(\infty)>0) \geq \alpha\} \subseteq [0,\infty)\times\Rb.\]
We think of $U_\alpha$ as the ``good'' part of the tube --- if a particle is born in $U_\alpha$ then it has probability at least $\alpha$ of contributing to $Z(\infty)$. Finally, for any particle $u$ and $t\geq0$, define
\[I_\alpha(u;t) = \int_0^{t\wedge S_u} \ind_{\{X_u(s) \in U_\alpha\}} ds;\]
$I_\alpha(u;t)$ is the time spent by particle $u$ in the set $U_\alpha$ before $t$.
\end{defn}

\begin{figure}[h!]
  \centering
      \includegraphics[width=\textwidth]{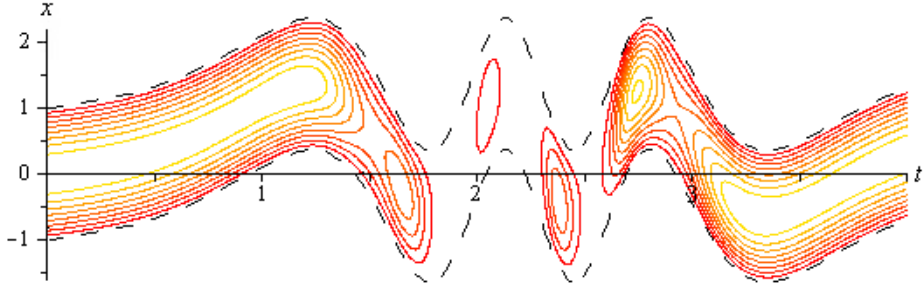}
  \caption{Approximation to a section of $U_\alpha$ for eight different values of $\alpha$ when \mbox{$f(t)=\sin(a \tanh(t+b))+c$} for some constants $a$, $b$ and $c$.}
\end{figure}

Our first task is to convert to using $\tilde f$ and $\tilde L$; the fact that $\tilde L$ is bounded will prove useful.

\begin{lem}\label{bdd_L}
The pair $(\tilde f,\tilde L)$ satisfies usual conditions (II, III, IV), and $\tilde S := S^{\tilde f,\tilde L}\geq S^{f,L}/2 >0$.
\end{lem}

\begin{proof}
We note that $\tilde L$ is twice continuously differentiable and hence so is $\tilde f$, and that $\tilde L(t) = L(t)$ whenever $L(t)\leq L_0$, $\tilde L(t) \geq L_0$ whenever $L(t)\geq L_0$, and $\tilde L(t) \leq L(t)\wedge(L_0+1)$ for all $t\geq0$. We first claim that $E^{\tilde f,\tilde L}(t) = o(t)$, working by comparison with $E^{f,L}$. Indeed, when $L(t)\leq L_0$ we clearly have $|\tilde L'(t)|=|L'(t)|$ and $|\tilde L''(t)|=|L''(t)|$. When $L(t)>L_0$,
\[\tilde L'(t) = L'(t)(1-2(L(t)-L_0)^2)e^{-(L(t)-L_0)^2}\]
so $|\tilde L'(t)|\leq |L'(t)|$. Also,
\begin{multline*}
\tilde L''(t) = L''(t)e^{-(L(t)-L_0)^2} - 6L'(t)^2(L(t)-L_0)e^{-(L(t)-L_0)^2}\\
 - 2L''(t)(L(t)-L_0)^2e^{-(L(t)-L_0)^2} + 4L'(t)^2 (L(t)-L_0)^3e^{-(L(t)-L_0)^2}
\end{multline*}
so (since the sizes of $xe^{-x^2}$, $x^2 e^{-x^2}$ and $x^3 e^{-x^2}$ are bounded above by 1)
\begin{multline*}
\int_0^t |\tilde L''(s)|\tilde L(s) ds \leq \int_0^t |L''(s)|L(s) ds + 6(L_0+1)\int_0^t L'(s)^2 ds\\
 + 2\int_0^t |L''(s)|L(s)ds + 4(L_0+1)\int_0^t L'(s)^2 ds.
 \end{multline*}
Each of these terms on the right-hand side above is $o(t)$ since
\[\int_0^t L'(s)^2 ds = L'(t)L(t) - L'(0)L(0) - \int_0^t L''(s)L(s) ds\]
and $L$ satisfies our usual conditions.
As $\tilde f'(t) = f'(t) + L'(t) - \tilde L'(t)$, and similarly for $\tilde f''$, we may also bound $|\tilde f'(t)|\tilde L(t)$ and $\int_0^t |\tilde f''(s)|\tilde L(s) ds$ simply by using the above estimates along with the triangle inequality and linearity of the integral. Thus, provided that \mbox{$E^{f,L}(t)=o(t)$} we must have $E^{\tilde f,\tilde L}(t)=o(t)$. Clearly also $S^{\tilde f, \tilde L} \in (-\infty,\infty)$.

Secondly, we claim that $\limsup_{t\to\infty} \frac{1}{t}\log L(t) \leq 0$. Suppose not; then there exist $\varepsilon>0$ and $t_n\to\infty$ such that $L(t_n) > e^{\varepsilon t_n}$ for each $n$. Setting
\[T_n:=\sup\{t\in[0,t_n) : L(t)< e^{\varepsilon t_n}/2 \},\]
if $T_n>0$ (which must occur for all but finitely many $n$) then by the mean value theorem we can choose $c_n\in(T_n,t_n)$ such that $L'(c_n) \geq e^{\varepsilon t_n}/2t_n$. But $L(c_n) \geq e^{\varepsilon t_n}/2$, so $L'(c_n)L(c_n) \geq e^{2\varepsilon t_n}/4t_n$, contradicting the assumption that $(f,L)$ satisfies the usual conditions (specifically the requirement that $L'(t)L(t)=o(t)$).

Thirdly, we show that $\int_0^t \tilde f'(s)^2 ds = \int_0^t f'(s)^2 ds + o(t)$. By Minkowski's inequality,
\begin{align*}
&\left(\int_0^t \tilde f'(s)^2 ds\right)^{1/2}\\
&= \left(\int_0^t (f'(s) + L'(s) - \tilde L'(s))^2 ds\right)^{1/2}\\
&\leq \left(\int_0^t f'(s)^2 ds\right)^{1/2} + \left(\int_0^t L'(s)^2 ds\right)^{1/2} + \left(\int_0^t \tilde L'(s)^2 ds\right)^{1/2}
\end{align*}
but
\[\int_0^t L'(s)^2 ds = L(t)L'(t) - L(0)L'(0) - \int_0^t L''(s) L(s) ds = o(t)\]
and the same calculation holds for $\tilde L$. Similarly by writing out $(\int_0^t f'(s)^2 ds)^{1/2}$ in terms of $\tilde f'$, $L'$ and $\tilde L$ and applying Minkowski's inequality we get that
\[\int_0^t f'(s)^2 ds \leq \int_0^t \tilde f'(s)^2 ds + o(t).\]

Our final claim is that $\tilde S := S^{\tilde f,\tilde L}\geq S^{f,L}/2 >0$. Indeed, using various facts just established,
\begin{align*}
&\frac{1}{t}\int_0^t \left(r - \frac{1}{2}\tilde f'(s)^2 - \frac{\pi^2}{8\tilde L(s)^2} + \frac{\tilde L'(s)}{\tilde L(s)}\right) ds\\
&\geq \frac{1}{t}\int_0^t \left(r - \frac{1}{2}f'(s)^2 - \frac{\pi^2}{8 L(s)^2}\right)ds - \frac{1}{t}\int_0^t \frac{\pi^2}{8L_0^2} ds + \frac{1}{t}\log \tilde L(t) - \frac{1}{t}\log \tilde L(0)\\
&\geq \frac{1}{t}\int_0^t \left(r - \frac{1}{2}f'(s)^2 - \frac{\pi^2}{8 L(s)^2}\right)ds - S/2 + \frac{1}{t}\log(L(t)\wedge 1) - \frac{1}{t}\log \tilde L(0)
\end{align*}
so that (since $\limsup \frac{1}{t}\log L(t) \leq 0$)
\begin{align*}
&\liminf_{t\to\infty}\frac{1}{t}\int_0^t \left(r - \frac{1}{2}\tilde f'(s)^2 - \frac{\pi^2}{8\tilde L(s)^2} + \frac{\tilde L'(s)}{\tilde L(s)}\right) ds\\
&\geq \liminf_{t\to\infty}\left\{\frac{1}{t}\int_0^t \left(r - \frac{1}{2}f'(s)^2 - \frac{\pi^2}{8\tilde L(s)^2}\right) ds + \frac{1}{t}\log L(t)\right\} - S/2\\
&\geq \liminf_{t\to\infty}\frac{1}{t}\int_0^t \left(r - \frac{1}{2}f'(s)^2 - \frac{\pi^2}{8 L(s)^2} + \frac{L'(s)}{L(s)}\right) ds - S/2\\
&= S^{f,L}/2
\end{align*}
as required.
\end{proof}

Our next lemma establishes that for sufficiently small $\alpha$, $U_\alpha$ --- which we think of as the good part of the tube --- stretches to near the top and bottom edges of the $L$-tube for almost $S/2r$ proportion of the time. To do this we use the identity given in Lemma \ref{extinction_lem} combined with the spine decomposition. For $\delta\in(0,1)$ and $t\geq0$, let
\[\hat L(t) := ((1-\delta) L(t)) \vee (L(t)-\delta).\]

\begin{lem}\label{tube_filling}
Fix $\delta \in(0,1)$ and $\beta<1$. If $S > 0$ then for sufficiently small $\alpha>0$ and large $T$, we have
\[\int_0^t \ind_{\{(s,x)\in U_\alpha \hsl \forall x \in [f(s)-\hat L(s), f(s)+\hat L(s)]\}} ds \geq \beta\frac{S}{2r}t \hs \forall t\geq T.\]
\end{lem}

\begin{proof}
Fix $q\in(0,\frac{1-\beta}{3})$ and $p\in(\beta+3q,1)$; we show that for
\[\alpha = \frac{q\tilde S \cos(\pi\delta/2)}{2re^{(L_0+1)(r\sqrt{2/q\tilde S}+1)}}\]
and all sufficiently large $t$ we have
\[\int_0^t \ind_{\{(s,x)\in U_\alpha \hsl \forall x \in [f(t)-\hat L(s), f(t)+\hat L(s)]\}} ds \geq (p-3q)\frac{S}{2r}t.\]
We begin working with $\tilde f$ and $\tilde L$; we shall move back to $f$ and $L$ towards the end of the proof.
Let
\[J_t = \inf_{s\geq t} \left\{\int_0^s \left(r - \frac{\pi^2}{8\tilde L(u)^2} - \frac{1}{2}\tilde f'(u)^2 + \frac{\tilde L'(u)}{2\tilde L(u)} - q\tilde S\right)du - E^{\tilde f, \tilde L}(s)\right\},\]
and define three subsets, $U$, $V$ and $W$, of $[0,\infty)$ by
\[U = \{t\geq0: J_t \hbox{ is increasing at } t\}, \hs V = \left\{t\geq 0 : |\tilde f'(t)|< r\sqrt{2/q\tilde S}\right\}\]
and
\[W = \{t\geq0 : |\tilde L'(t)|\leq 1\}.\]
If $J$ is increasing at $t$, then clearly for any $s>0$
\begin{multline*}
\int_0^{t+s}\left( r - \frac{\pi^2}{8\tilde L(u)^2} - \frac{1}{2}\tilde f'(u)^2 + \frac{\tilde L'(u)}{2\tilde L(u)} - q\tilde S\right) du - E^{\tilde f, \tilde L}(t+s)\\
> \int_0^t \left( r - \frac{\pi^2}{8\tilde L(u)^2} - \frac{1}{2}\tilde f'(u)^2 + \frac{\tilde L'(u)}{2\tilde L(u)} - q\tilde S\right) du - E^{\tilde f, \tilde L}(t),
\end{multline*}
and hence
\[\int_t^{t+s} \left(r - \frac{\pi^2}{8\tilde L(u)^2} - \frac{1}{2}\tilde f'(u)^2 + \frac{\tilde L'(u)}{2\tilde L(u)}\right) du - E^{\tilde f,\tilde L}(t+s) + E^{\tilde f, \tilde L}(t) > q\tilde S s.\]
Thus if $t\in U\cap V\cap W$ then, as in Proposition \ref{uiprop}, we can apply the spine decomposition and Lemma \ref{unscaled_int_lem} to get, for any $x\in (-\tilde L(t), \tilde L(t))$,
\begin{align*}
\Qt_x^{\tilde f_t,\tilde L_t}[Z^{\tilde f_t,\tilde L_t}(\infty) | \Gg_\infty] &= \int_0^\infty 2r e^{-rs}\zeta^{\tilde f_t,\tilde L_t}(s) ds + \lim_{t\to\infty}e^{-rt}\zeta^{\tilde f_t,\tilde L_t}(t)\\
&\leq \int_0^\infty 2r e^{-\int_0^s(r-\frac{\pi^2}{8\tilde L_t(u)^2} - \frac{1}{2}\tilde f_t'(u)^2 + \frac{\tilde L_t'(u)}{2\tilde L_t(u)})du + E^{\tilde f_t,\tilde L_t}(s)} ds\\
&\leq \int_0^\infty 2r e^{-\int_t^{t+s}(r-\frac{\pi^2}{8\tilde L(u)^2} - \frac{1}{2}\tilde f'(u)^2 + \frac{\tilde L'(u)}{2\tilde L(u)})du}\\
&\hspace{20mm} \cdot e^{E^{\tilde f,\tilde L}(t+s) - E^{\tilde f,\tilde L}(t) + |\tilde f'(t)|\tilde L(t) + \frac{1}{2}|\tilde L'(t)|\tilde L(t)} ds\\
&\leq e^{|\tilde f'(t)|\tilde L(t) + \frac{1}{2}|\tilde L'(t)|\tilde L(t)} \int_0^\infty 2r e^{-q\tilde S s} ds\\
&\leq \frac{2r}{q\tilde S}e^{(r\sqrt{2/q\tilde S} + 1/2)(L_0+1)}
\end{align*}
Using the identity from Lemma \ref{extinction_lem} together with Jensen's inequality gives that for any $x\in[\tilde f(t)-(((1-\delta)\tilde L(t))\vee(\tilde L(t)-\delta)), \tilde f(t)+((1-\delta)\tilde L(t))\vee(\tilde L(t)-\delta)]$,
\begin{align*}
&\Pb_x(Z^{\tilde f_t,\tilde L_t}(\infty)>0)\\
&= \Qb_x^{\tilde f_t,\tilde L_t}\left[\frac{Z^{\tilde f_t,\tilde L_t}(0)}{Z^{\tilde f_t,\tilde L_t}(\infty)}\right]\\
&\geq \Qt_x^{\tilde f_t,\tilde L_t}\left[ \Qt_x^{\tilde f_t,\tilde L_t}\left[\left. \frac{1}{Z^{\tilde f_t,\tilde L_t}(\infty)} \right| \Gg_\infty \right]\right] e^{-\frac{1}{2}\tilde L'(t)\tilde L(t)}\cos\left(\frac{\pi x}{2\tilde L(t)}\right) \\
&\geq \Qt_x^{\tilde f_t,\tilde L_t}\left[  \frac{1}{\Qt_x^{\tilde f_t,\tilde L_t}[Z^{\tilde f_t,\tilde L_t}(\infty)|\Gg_\infty]} \right] e^{-\frac{1}{2} L_0+1}\cos\left(\frac{\pi (L_0+1-\delta)}{2(L_0+1)}\right)\\
&\geq \frac{q\tilde S}{2r e^{(r\sqrt{2/q\tilde S} + 1)(L_0+1)}}\cos\left(\frac{\pi(L_0+1-\delta)}{2(L_0+1)}\right).
\end{align*}
Now, since
\begin{multline*}
[\tilde f(t)-(((1-\delta)\tilde L(t))\vee(\tilde L(t)-\delta)), \tilde f(t)-(((1-\delta)\tilde L(t))\vee(\tilde L(t)-\delta))]\\
\supseteq [f(t) + L(t)-\tilde L(t)-\hat L(t), f(t)+\hat L(t)]
\end{multline*}
we have shown that if $t\in U\cap V\cap W$ then $\Pb_x(Z^{f_t,L_t}(\infty)>0)$ is large enough for all $x\in[f(t)+L(t)-\tilde L(t)-\hat L(t), f(t)+\hat L(t)]$. If $x\in[f(t),f(t)+L(t)-\tilde L(t)-\hat L(t))$ then running the same argument as above but using $\tilde f^{(x)}(s):= \tilde f(s) - \tilde f(0) + x$, $s\geq0$ in place of $\tilde f$ gives exactly the same result: so we have that $\Pb_x(Z^{f_t,L_t}(\infty)>0)$ is large enough for the half-region $[f(t),f(t)+\hat L(t)]$ and by symmetry for the whole region $[f(t)-\hat L(t),f(t)+\hat L(t)]$. Hence it now suffices to show that for large $t$,
\[\int_0^t \ind_{U\cap V\cap W}(s) ds \geq (p-3q)\frac{S}{2r}t.\]
But for all large enough $t$, since $J$ increases at rate at most $r$ (recall that $\int_0^t \frac{\tilde L'(s)}{2\tilde L(s)} ds = \log \tilde L(t) - \log \tilde L(0)$, which is bounded) and $\lim_{t\to\infty}J_t = (1-q)\tilde S$,
\[(p-q)\tilde S t \leq J_t \leq \int_0^t r\ind_U(s) ds.\]
Also, for large enough $t$ we must have $\int_0^t \tilde f'(s)^2 ds \leq 2rt$ (otherwise $\tilde S$ would be negative). Thus for large $t$
\[2rt \geq \int_0^t \tilde f'(s)^2 ds \geq \int_0^t \frac{2r^2}{q\tilde S} \ind_{V^c}(s) ds;\]
finally,
\[\int_0^t \tilde L'(s)^2 ds = \tilde L(t)\tilde L'(t) - \tilde L(0)\tilde L'(0) + \int_0^t \tilde L(s)\tilde L''(s) ds\]
so since $E^{\tilde f,\tilde L}=o(t)$ we have (again for large $t$)
\[\int_0^t \ind_{W^c}(s) ds \leq \int_0^t \tilde L'(s)^2 ds \leq \frac{q\tilde S}{r}t.\]
Hence for all large $t$,
\begin{align*}
\int_0^t \ind_{U\cap V\cap W}(s) ds &\geq \int_0^t \ind_U(s) ds - \int_0^t \ind_{V^c}(s)ds - \int_0^t \ind_{W^c}(s)ds\\
&\geq (p-q)\frac{\tilde S}{r}t - q\frac{\tilde S}{r}t - q\frac{\tilde S}{r}t \geq (p-3q)\frac{S}{2r}t
\end{align*}
as required.
\end{proof}

We now show that if a particle has remained in the tube for a long time, then it is very likely to have spent a long time in $U_\alpha$. The idea is that if $U_\alpha$ stretches to within $\delta$ of the edge of the tube for a proportion of time, then in order to stay out of $U_\alpha$ a particle must spend a long time in a tube of radius $\delta$. We use simple estimates for the time spent by Brownian motion in such a tube and apply these to our problem via the many-to-one theorem (Theorem \ref{many_to_one}).

\begin{lem}\label{local_time_lem}
For any $\delta>0$ and $k>0$,
\[\Pt\left(\int_0^t \ind_{\{\xi_s \in (-\delta, \delta)\}} ds > k\right) \leq 3e^{t/2 - k/4\delta}.\]
\end{lem}

\begin{proof}
We first claim that if we define $h_\delta: \Rb\to\Rb$ by
\[h_\delta(x) := \left\{\begin{array}{ll} |x| & \hbox{if } |x|\geq \delta \\ \frac{\delta}{2} + \frac{x^2}{2\delta} & \hbox{if } |x|<\delta \end{array}\right.\]
then
\[h_\delta(\xi_t) = \frac{\delta}{2} + \int_0^t h'_\delta (\xi_s) d\xi_s + \frac{1}{2\delta}\int_0^t \ind_{\{\xi_s\in(-\delta,\delta)\}} ds.\]
We check, by approximation with $C^2$ functions, that It\^o's formula holds for $h_\delta$. Define a function $g_{\delta,n}\in C^2(\Rb)$ for each $n\in\mathbb{N}$ by setting
\[g_{\delta,n}''(s) = \left\{\begin{array}{ll} 0 & \hbox{if } |x|\geq \delta \\ \frac{n}{\delta}(\delta-|x|) & \hbox{if } \delta-\frac{1}{n}<|x|<\delta \\ \frac{1}{\delta} & \hbox{if } |x|<\delta - \frac{1}{n} \end{array}\right.\]
with $g_{\delta,n}'(0)=0$, $g_{\delta,n}(0)=\delta/2$. Since $g\in C^2$, It\^o's formula tells us that
\[g_{\delta,n}(\xi_t) = g_{\delta,n}(\xi_0) + \int_0^t g_{\delta,n}'(\xi_s)d\xi_s + \frac{1}{2}\int_0^t g_{\delta,n}''(\xi_s)ds.\]
Since $g_{\delta,n}''\to h_\delta''$ Lebesgue-almost everywhere, by bounded convergence
\[\int_0^t g_{\delta,n}''(\xi_s) ds \to \int_0^t h_\delta''(\xi_s) ds \hs \Pt\hbox{-almost surely},\]
and $g_{\delta,n}\to h_\delta$ uniformly so for each $t$, $g_{\delta,n}(\xi_t) \to h_\delta(\xi_t)$ $\Pt$-almost surely. Also, by the It\^o isometry
\[\Pt\left[\left(\int_0^t (g_{\delta,n}'(\xi_s) - h_\delta'(\xi_s))d\xi_s\right)^2\right] = \Pt\left[\int_0^t (g_{\delta,n}'(\xi_s) - h_\delta'(\xi_s))^2 ds\right];\]
since $g_{\delta,n}'\to h_\delta'$ uniformly, the right hand side above converges to zero, and hence
\[\int_0^t g_{\delta,n}'(\xi_s)d\xi_s \to \int_0^t h_\delta'(\xi_s)d\xi_s \hs \Pt\hbox{-almost surely}.\]
Thus It\^o's formula does indeed hold for $h_\delta$, and since
\[\frac{1}{2}\int_0^t f_\delta''(s)ds = \frac{1}{2\delta}\int_0^t \ind_{\{\xi_s\in(-\delta,\delta)\}} ds\]
our claim holds. Now recall that under $\Pt$, the spine's motion is simply a Brownian motion, so
\[\Pt[e^{-\int_0^t h'_\delta(\xi_s)d\xi_s}] \leq \Pt[e^{-\int_0^t h'_\delta(\xi_s)d\xi_s - \frac{1}{2}\int_0^t h'_\delta(\xi_s)^2 ds}]e^{t/2} \leq e^{t/2}.\]
Thus
\begin{align*}
\Pt\left(\int_0^t \ind_{\{\xi_s \in (-\delta,\delta)\}} ds > k\right) &= \Pt\left(h_\delta(\xi_t) - \frac{\delta}{2} - \int_0^t h_\delta '(\xi_s) d\xi_s > \frac{k}{2\delta}\right)\\
&\leq \Pt\left(|\xi_t| - \int_0^t h_\delta ' (\xi_s) d\xi_s > \frac{k}{2\delta}\right)\\
&\leq \Pt\left(|\xi_t| > \frac{k}{4\delta}\right) + \Pt\left(-\int_0^t h_\delta ' (\xi_s) d\xi_s > \frac{k}{4\delta}\right)\\
&\leq \Pt\left[e^{|\xi_t|}\right]e^{-k/4\delta} + \Pt\left[e^{- \int_0^t h_\delta ' (\xi_s) d\xi_s}\right]e^{-k/4\delta}\\
&\leq 3e^{t/2 - k/4\delta},
\end{align*}
establishing the result.
\end{proof}

\begin{lem}\label{spine_U}
Fix $\beta<1$ and $\gamma>0$. If $S >0$ then for sufficiently small $\alpha>0$ and large $T$, we have
\[\Pb(\exists u \in\hat N(t) : I_\alpha(u;t)< \beta\frac{S}{2r}t) \leq e^{-\gamma t} \hs\hs \forall t\geq T.\]
\end{lem}

\begin{proof}
For any $\delta\in(0,1)$, by Lemma \ref{tube_filling} we may choose $\alpha>0$ and $T$ such that
\[\int_0^t \ind_{\{(s,x)\in U_\alpha \hsl \forall x\in[f(t)-L+\delta,f(t)+L-\delta]\}} ds \geq \left(\frac{1+\beta}{2}\right)\frac{S}{2r}t \hs \forall t\geq T.\]
Then if the spine particle is to have spent less than $\beta\frac{S}{2r}t$ time in $U_\alpha$ (yet remained within the tube of width $L$) then it must have spent at least $(\frac{1-\beta}{2})\frac{S}{2r}t$ within $\delta$ of the edge of the tube (provided that $t$ is large enough). That is, for $t\geq T$, if we let
\[V^1_s:=(f(s)-L(s), f(s)-L(s)+\delta)\cup(f(s)+L(s)-\delta, f(s)+L(s))\]
then
\begin{multline*}
\Pt\left(\xi_t \in \hat N(t), I_\alpha(\xi_t;t)<\beta\frac{S}{2r}t\right)\\
\leq \Pt\left(\xi_t\in\hat N(t), \int_0^t \ind_{\{\xi_s \in V^1_s\}}ds > \left(\frac{1-\beta}{2}\right)\frac{S}{2r}t\right).
\end{multline*}
In fact, using the fact that if $\xi_t \in \hat N(t)$ then we may apply two simple Girsanov measure changes and our usual estimates on them. The first will give the spine drift $f'$, and the second will give it an extra drift $L'$. Letting
\[V^2_s:=(-L(s), -L(s)+\delta)\cup(L(s)-\delta,L(s))\]
we have
\begin{align*}
\Pt&\bigg(\xi_t \in \hat N(t), I_\alpha(\xi_t;t)<\beta\frac{S}{2r}t\bigg)\\
&\leq \Pt\left[\frac{\ind_{\{|\xi_s|<L(s)\hsl\forall s\in[0,t]\}}}{e^{\int_0^t f'(s)d\xi_s - \frac{1}{2}\int_0^t f'(s)^2 ds}}\ind_{\{\int_0^t \ind_{\{\xi_s \in V^2_s\}}ds > (\frac{1-\beta}{2})\frac{S}{2r}t\}}\right]\\
&\leq e^{|f'(t)|L(t)+\int_0^t |f''(s)|L(s)ds}\\
&\hspace{10mm}\cdot\Pt\left(|\xi_s|<L(s)\hsl\forall s\in[0,t], \int_0^t \ind_{\{\xi_s \in V^2_s\}}ds > \left(\frac{1-\beta}{2}\right)\frac{S}{2r}t\right)\\
&\leq 2e^{|f'(t)|L(t)+\int_0^t |f''(s)|L(s)ds}\\
&\hspace{10mm}\cdot\Pt\left(|\xi_s|<L(s)\hsl\forall s\in[0,t], \int_0^t \ind_{\{\xi_s \in (L(s)-\delta,L(s))\}}ds > \left(\frac{1-\beta}{2}\right)\frac{S}{4r}t\right)\\
&\leq 2e^{|f'(t)|L(t)+\int_0^t |f''(s)|L(s)ds}\\
&\hspace{10mm}\cdot\Pt\left[\frac{\ind_{|\xi_s|<2L(s)\hsl\forall s\in[0,t]}}{e^{\int_0^t L'(s)d\xi_s - \frac{1}{2}\int_0^t L'(s)^2 ds}}\ind_{\{\int_0^t \ind_{\{\xi_s \in (L(s)-\delta,L(s))\}}ds > \left(\frac{1-\beta}{2}\right)\frac{S}{4r}t\}}\right]\\
&\leq 2e^{|f'(t)|L(t)+\int_0^t |f''(s)|L(s)ds + 2|L'(t)|L(t)+2\int_0^t |L''(s)|L(s)ds}\\
&\hspace{10mm}\cdot\Pt\left(\int_0^t \ind_{\{\xi_s \in (-\delta,0)\}}ds > \left(\frac{1-\beta}{2}\right)\frac{S}{4r}t\right).
\end{align*}
Using the estimate given in Lemma \ref{local_time_lem}, and usual condition (III), we get that for large enough $t$
\[\Pt\left(\xi_t \in \hat N(t), I_\alpha(\xi_t;t)<\beta\frac{S}{2r}t\right) \leq e^{(r+1)t - \frac{1}{4\delta}\left(\frac{1-\beta}{2}\right)\frac{S}{4r}t}.\]
Finally, taking $\delta = \frac{(1-\beta)S}{32r(2r+\gamma+1)}$ and using the many-to-one theorem (Theorem \ref{many_to_one}), for large $t$
\[\Pt\left(\exists u \in \hat N(t) : I_\alpha(u;t)<\beta\frac{S}{2r}t\right) \leq e^{rt}\Pt\left(\xi_t \in \hat N(t), I_\alpha(\xi_t;t)<\beta\frac{S}{2r}t\right) \leq e^{-\gamma t}.\qedhere\]
\end{proof}
We now combine the above results to achieve the aim of this section.

\begin{prop}\label{invspineprop}
Recall that $\Upsilon$ is the extinction time for the process. If $S > 0$ then
\[\Pb(\Upsilon = \infty) = \Pb(Z(\infty)>0).\]
\end{prop}

\begin{proof}
We note that $\{Z(\infty)>0\} \subseteq \{\Upsilon = \infty\}$, so it suffices to show that for any $\varepsilon >0$,
\[\Pb(\Upsilon = \infty, \hs Z(\infty)=0) < \varepsilon.\]
To this end, fix $\varepsilon > 0$ and choose $\alpha$ small enough and $T_0$ large enough that
\[\Pb(\exists u \in\hat N(t) : I_\alpha(u;t)< \frac{S}{4r}t) < \varepsilon/3 \hs \forall t\geq T_0\]
(this is possible by Lemma \ref{spine_U}). Now choose an integer $m$ large enough such that \mbox{$(1-\alpha)^m < \varepsilon/3$}. Finally, choose $T\geq T_0$ large enough that
\[\sum_{j=0}^{m-1} \frac{e^{-S T / 4}(S T / 4)^j}{j!} < \varepsilon/3.\]
Then
\begin{align*}
\Pb(\Upsilon = \infty, \hs Z(\infty)=0) &\leq \Pb(\exists u \in \hat N(T), \hs Z(\infty)=0)\\
&< \Pb\left(\exists u \in \hat N(T), \hsl I_\alpha(u;T) \geq \frac{S}{4r}T, \hsl Z(\infty)=0\right) + \varepsilon/3.
\end{align*}
Now, if a particle $u$ has spent at least $\frac{S}{4r}T$ time in $U_\alpha$ then (by the choice of $T$, since the births along $u$ form a Poisson process of rate $r$) it has probability at least $(1-\varepsilon/3)$ of having at least $m$ births whilst in $U_\alpha$. Each of these particles born within $U_\alpha$ launches an independent population from a point $(t,x)\in U_\alpha$, so that
\[Z(\infty)\geq \sum_{v<u} e^{-r (S_v-\sigma_v)} Z_v(\infty) \ind_{\{(S_v-\sigma_v, X_u(S_v-\sigma_v))\in U_\alpha\}}\]
where each $Z_v$ is a non-negative martingale on the interval $[S_v-\sigma_v,\infty)$ with law equal to that of $Z^{f_t,L_t}$ started from $x$ for some $(t,x)\in U_\alpha$, and hence satisfying \mbox{$\Pb(Z_v(\infty)>0) \geq \alpha$}. Thus
\begin{align*}
&\Pb(\Upsilon = \infty,\hsl Z(\infty)=0)\\
&\leq \Pb\left(\exists u \in \hat N(T), \hs I_\alpha(u;T) \geq \frac{S}{4r}T, \hs Z(\infty)=0\right) + \varepsilon/3\\
&\leq \Pb\left(\exists u \in \hat N(T), \hs \left\{\begin{array}{c} u \hbox{ has had at least}\\ m \hbox{ births within } U_\alpha\end{array}\right\}, \hs Z(\infty)=0\right) + 2\varepsilon/3\\
&\leq (1-\alpha)^m + 2\varepsilon/3 \hs < \hs \varepsilon
\end{align*}
which completes the proof.
\end{proof}

We draw our results together as follows.

\begin{mainprf}
All that remains is to combine Proposition \ref{uiprop} with Corrolary \ref{lim_cor} to gain the desired growth bounds; Proposition \ref{invspineprop} guarantees that we are working on the correct set.\qed
\end{mainprf}

\section{Extending the class of functions}\label{extensions}
As promised, we can extend Theorem \ref{mainthm} to cover more general subsets of $C[0,\infty)$ in an obvious way: if a set $B\subset C[0,\infty)$ is contained within (or contains) an $L$-tube about a function $f$, then the set of particles with paths in $B$ is a subset (respectively, superset) of the set of particles with paths within $L$ of $f$, and if $(f,L)$ satisfies our usual conditions then we have an immediate upper (lower) bound on the number of particles within $B$. That is, for any $B\subset C[0,\infty)$,
\begin{equation}\label{ext1}
\sup \Pb(\hat N^{f,L}(t)\neq\emptyset) \leq \Pb(\hat N^B(t)\neq\emptyset) \leq \inf \Pb(\hat N^{f,L}(t)\neq\emptyset)
\end{equation}
and
\begin{equation}\label{ext2}
\sup |\hat N^{f,L}(t)|  \leq |N^B(t)| \leq \inf |\hat N^{f,L}(t)|
\end{equation}
where both suprema are taken over all $f$ and $L$ such that $(f,L)$ satisfies our usual conditions and
\[\{g\in C[0,\infty) : |g(s)-f(s)| < L(s) \hs\forall s\in[0,\infty)\} \subseteq B,\]
both infima are taken over all $f$ and $L$ such that $(f,L)$ satisfies our usual conditions and
\[B \subseteq \{g\in C[0,\infty) : |g(s)-f(s)| < L(s) \hs\forall s\in[0,\infty)\},\]
and
\[N^B(t):= \{u\in N(t) : \exists g\in B \hbox{ with } X_u(s)=g(s) \hs \forall s\in[0,t]\}.\]
The obvious question now is whether this allows us to give growth rates for all sets in $C[0,\infty)$. The answer is no: there are still some seemingly reasonable sets that are not covered (which we shall see shortly).

Thus the natural question becomes whether we can instead characterise, in a more succinct way, the class of functions that Theorem \ref{mainthm} does cover, subject to using the extensions provided by (\ref{ext1}) and (\ref{ext2}). Can we weaken our usual conditions in some way that we can easily write down? The answer again seems to be, more or less, no. We may drop condition (I) as our eventual growth rate does not depend on the initial position of the particle as long as there is a path within our set that starts at the same point as the initial position of the first particle. We may also effectively drop condition (IV) --- since it is not possible to get $S=\infty$ without violating condition (III), and the case $S=-\infty$ can always be covered either by bounding above using (\ref{ext1}) and (\ref{ext2}) or by using the many-to-one theorem, Theorem \ref{many_to_one}, more directly. However the interesting conditions (II) and (III) are difficult to shake off, a fact which is best demonstrated by a series of examples.

It is easiest to first consider condition (III).

\begin{ex}\label{badex1}
Take $L(t)\equiv L>0$ to be constant, and let
\[f_\delta(t):=\delta\sin(t/\delta);\]
then as $\delta\to0$, $f_\delta$ converges uniformly to the zero function, $f(t)\equiv 0$.
By Theorem \ref{mainthm} we know that on survival,
\[\lim_{t\to\infty} \frac{1}{t}\log|\hat N^{f,L}(t)| = r - \frac{\pi^2}{8L^2}.\]
However, if the result of Theorem \ref{mainthm} held for each $f_\delta$ then by approximation via (\ref{ext1}) and (\ref{ext2}) we would have (on survival)
\[\lim_{t\to\infty} \frac{1}{t}\log|\hat N^{f,L}(t)| = r - \frac{\pi^2}{8L^2} - \frac{1}{4}.\]
Of course, $(f_\delta,L)$ does not satisfy usual condition (III) and hence this contradiction does not appear -- but the example shows that we cannot simply drop the requirement that $\int_0^t |f''(s)|L(s) ds = o(t)$.
\end{ex}

\begin{ex}\label{badex2}
Take $f(t)\equiv 0$ and $L(t)= 2 + \sin(t^{3/2})$. Intuitively, the sine term oscillates so fast for large $t$ that we are effectively constrained within a tube of constant width 1. Thus we expect (and it is not too hard to imagine a hands-on proof using Theorem \ref{mainthm}) that we should have a growth rate of $r - \pi^2/8$. However, one may show (for example by using the periodicity of sine and approximating the integral by a sum) that
\[\int_0^t \frac{1}{L(s)^2}ds \lesssim \frac{2t}{3\sqrt 3}\]
so that if the result of Theorem \ref{mainthm} held in this case we would have a growth rate of at least $r - \pi^2/12\sqrt 3$. Again, $(f,L)$ does not satisfy usual condition (III) and we see that we cannot just drop the requirement that $\int_0^t |L''(s)|L(s) ds = o(t)$.
\end{ex}

\begin{ex}\label{badex3}
Take $f_0(t)\equiv 0$, $L_0(t) = \sqrt t$, $f_1(t) = t$ and $L_1(t) = t + \sqrt t$. Then the growth rate for $(f_0,L_0)$ is $r$; and since the $L_0$-tube about $f_0$ is contained in the $L_1$-tube about $f_1$, we must have a growth rate for $(f_1,L_1)$ of at least $r$ (in fact it is exactly $r$ since it is well-known that the growth rate of the entire system is $r$). If the result of Theorem \ref{mainthm} held for $(f_1,L_1)$ then its growth rate would be $r - 1/2$; so we see that we cannot simply drop the condition that $|f'(t)|L(t) + |L'(t)|L(t) = o(t)$.
\end{ex}

Now consider condition (II). We can approximate any continuous function with twice continuously differentiable functions, but then how do we approach the conditions on the second derivative (from condition (III))? Even for constant $L$, there are some nowhere-differentiable paths $f$ such that we may find a growth rate for $\hat N^{f,L}$ using (\ref{ext1}) and (\ref{ext2}), and some for which we may not. The lack of even a first derivative to work with in these cases precludes the existence of an obvious simple condition to tell us where to draw the line between these two groups.
We claim simply that any non-smooth sets are best considered on a case-by-case basis using Theorem \ref{mainthm} together with (\ref{ext1}) and (\ref{ext2}).

For example, again with constant $L$, we may easily (by approximating by its partial sums) give a growth rate for the function
\[f(t) = \sum_{n=0}^\infty a^n(\cos(b^n\pi\log(t+1))-1)\]
(where $b$ is a positive odd integer, $0<a<1$ and $ab>1+3\pi/2$), which is a time change of a Weierstrass function and hence, by the chain rule, nowhere differentiable. On the other hand we cannot give an exact growth rate along (almost) any given Brownian path: any uniformly approximating functions must (by the fact that Brownian motion has independent increments) violate our conditions on the second derivative of $f$ in (III).

\section{The critical case $S=0$}\label{critical_sec}
Of course, it is also possible to ask what happens when $S=0$, although as we stated in Section \ref{beta_sec}, we are unable to give a general theory. We did, however, state two results in Section \ref{beta_sec} as examples of what may be achieved by adjusting our earlier methods, and we prove those now.

\begin{proof}[Proof of Theorem \ref{betaex}]
In the case $\beta<1/3$ we may simply mimic the requisite part of the proof of Proposition \ref{uiprop}, using the fact that for $\beta<1/3$,
\[\int_0^t \left(r-\frac{1}{2}f'(s)^2 -\frac{\pi^2}{8L(s)^2} + \frac{L'(s)}{2L(s)}\right) ds = \frac{\pi^2}{8\gamma^2(1-2\beta)}(t+1)^{1-2\beta} + o(t^{1-2\beta})\]
and
\[E(t) = \gamma\sqrt{2r}(t+1)^\beta + o(t^\beta).\]
Now suppose that $\beta>1/3$. We proceed in very much the same way as in the main part of the article, leaving out many of the details. Direct calculation reveals that for $\beta>1/3$,
\[\int_0^t (r-\frac{1}{2}f'(s)^2 -\frac{\pi^2}{8L(s)^2} + \frac{L'(s)}{2L(s)}) ds = \alpha\sqrt{2r}(t+1)^\beta + o(t^\beta)\]
and
\[E(t) = \gamma\sqrt{2r}(t+1)^\beta + o(t^\beta).\]
Thus, by the spine decomposition,
\[\Qt[Z(t)|\Gg_\infty] \leq \int_0^t 2r e^{-(\alpha-\gamma)\sqrt{2r}(s+1)^\beta + o(s^\beta)} ds + e^{-(\alpha-\gamma)\sqrt{2r}(t+1)^\beta + o(t^\beta)}\]
which converges as $t\to\infty$ provided that $\alpha > \gamma$. We deduce that $\Pb(Z(\infty)>0)>0$ provided that $\alpha > \gamma$, and indeed for all $\alpha$ and $\gamma$ since for fixed $\alpha$, increasing $\gamma$ can only increase the probability of survival. The same argument as Proposition \ref{limsup_prop} gives
\[\limsup_{t\to\infty} \frac{\log|\hat N(t)|}{t^\beta} \leq (\alpha + \gamma)\sqrt{2r}.\]

Now, take $\varepsilon>0$ and define $\tilde f(t) := f(t) - (\gamma-\varepsilon)L(t)$ and $\tilde L(t) := \varepsilon L(t)$. Note that the $(\tilde f, \tilde L)$-tube is contained within the $(f,L)$-tube. Define
\[W(t):= \sum_{u\in \hat N^{f,L}(t)} e^{-rt} G_u^{\tilde f, \tilde L}(t)\]
and note that the same argument as in Proposition \ref{liminf_prop} gives that on $\{\liminf W(t)>0\}$ we have
\[\liminf_{t\to\infty} \frac{\log|\hat N(t)|}{t^\beta} \geq (\alpha + \gamma - \varepsilon)\sqrt{2r}.\]
Thus it suffices to show that $\{\liminf W(t)>0\}$ agrees with $\{\Upsilon^{f,L}=\infty\}$ up to a set of zero probability.

Following Lemma \ref{spine_U} and Proposition \ref{invspineprop}, we see that in fact it suffices to show that for any $\delta>0$ we can bound from below the probability that a particle in $\hat N^{f,L}(t)$ which is not within $\delta$ of the edge of the $(f,L)$-tube contributes something positive to $\liminf W(t)$, in analogy with Lemma \ref{tube_filling}. But $W(t) \geq Z^{\tilde f, \tilde L}(t)$, and so instead we show that a particle in $\hat N^{f,L}(t)$ which is not within $\delta$ of the edge of the $(f,L)$-tube contributes something positive to $Z^{\tilde f,\tilde L}(\infty)$.

Now (possibly subject to decreasing $\varepsilon$ further, but this is no problem) we may use the argument given in Lemma \ref{tube_filling} to show that for small enough $\alpha'$ the set $U_{\alpha'}$ for $(\tilde f, \tilde L)$ stretches to near the top and bottom of the $(f,L)$ tube: even when we are distance $\delta$ from the top edge of the tube at time $T$, the smaller tube with radius $\varepsilon (t+1)^\beta$ about $\sqrt{2r}t - \alpha(t+1)^\beta + \gamma(T+1)^\beta - \delta$ fits (for all times $t\geq T$) within the tube of radius $L$ about $f$. Then by using the spine decompositon and Jensen's inequality as in Proposition \ref{uiprop}, we can bound the probability of contributing to $Z^{\tilde f,\tilde L}(\infty)$ away from zero (over all $T$). We may take the same approach when starting from a position closer to the centre of the tube (that is, further than $\delta$ from the edge). Thus, for small enough $\alpha'$, $U_{\alpha'}$ for $(\tilde f,\tilde L)$ stretches to within $\delta$ of the edge of the $(f,L)$ tube for \emph{all} times $t\geq0$. By the argument above, this is enough to complete the proof as in Lemma \ref{spine_U} and Proposition \ref{invspineprop}.
\end{proof}


\begin{proof}[Proof of Theorem \ref{thirdex}]
The first part of the proof proceeds exactly as that of Theorem \ref{betaex}, but with
\[\int_0^t (r - \frac{1}{2}f'(s)^2 -\frac{\pi^2}{8L(s)^2} + \frac{L'(s)}{2L(s)}) ds = \left(\alpha\sqrt{2r}-\frac{3\pi^2}{8\gamma^2}\right)(t+1)^{1/3} + o(t^{1/3})\]
and
\[E(t) = \gamma \sqrt{2r} (t+1)^{1/3} + o(t^{1/3}):\]
the spine decomposition converges if
\[-\alpha\sqrt{2r} + \frac{3\pi^2}{8\gamma^2} + \gamma\sqrt{2r} < 0,\]
so $\Pb(Z(\infty)>0)>0$ if
\[\alpha > \gamma + \frac{3\pi^2}{8\gamma^2\sqrt{2r}}.\]
But increasing $\gamma$ makes the right-hand side of this inequality larger as soon as $\gamma\geq\gamma_1$, and increasing $\gamma$ can only make $\Pb(Z(\infty)>0)$ larger, so (after some rearrangements) we deduce that $\Pb(Z(\infty)>0)>0$ provided \emph{either} $\gamma\geq\gamma_1$ and $\alpha>3\gamma_1/2$ \emph{or} $\gamma<\gamma_1$ and $\alpha>\gamma+\frac{3\pi^2}{8\gamma^2\sqrt{2r}}$.

Under $\Qb$, $Z(t)$ diverges to infinity if $-\alpha\sqrt{2r} + \frac{3\pi^2}{8\gamma^2} - \gamma\sqrt{2r} > 0$. Since $\alpha > 0$, this is impossible if $\gamma \geq \gamma_0$; so we need $\gamma<\gamma_0$ and $\alpha<\frac{3\pi^2}{8\gamma^2\sqrt{2r}}-\gamma$. If $Z(t)\to\infty$ almost surely under $\Qb$, then by Lemma \ref{stdmeas}, $Z(t)\to0$ almost surely under $\Pb$.

The calculations of the $\liminf$s and $\limsup$s are standard, as in Propositions \ref{uiprop}, \ref{liminf_prop} and \ref{limsup_prop}. However, we must again take a different approach to show that $\{Z(\infty)>0\}$ agrees with $\{\Upsilon=\infty\}$ up to a set of zero probability. Our proof, below, is specially adapted to this particular case and takes advantage of the convenient --- and well-known --- fact that $\frac{1}{3} + 2\times\frac{1}{3} = 1$.

We can easily show, straight from the spine decomposition and as in previous calculations, that for any $\delta\in(0,\gamma/2)$, there exists $\alpha'>0$ such that $U_{\alpha'}$ stretches to within $\delta t^{1/3}$ of the edges of the tube at time $t$ for any $t>0$. Thus (in analogy with Lemma \ref{spine_U}) we would like to show, loosely speaking, that with high probability, particles spend a long time outside the tubes of radius $\delta (s+1)^{1/3}$, $s\in[0,t]$ nested just inside the upper and lower boundaries of our main tube about $f$. The idea is that if particles do not want to leave $\hat N(t)$ then staying near the boundaries of the tube is a bad tactic. To be more precise about this, following the direction of part of the proof of Lemma \ref{spine_U} and setting
\[V^1_s:=(f(s)-L(s), f(s)-L(s)+\delta(s+1)^{1/3})\cup(f(s)+L(s)-\delta(s+1)^{1/3}, f(s)+L(s))\]
and
\[V^2_s:=(-L(s), -L(s)+\delta(s+1)^{1/3})\cup(L(s)-\delta(s+1)^{1/3},L(s))\]
we have
\begin{align*}
&\Pt\bigg(\xi_t \in \hat N(t), I_\alpha(\xi_t;t)< t/2\bigg)\\
&\leq \Pt\left(\xi_t\in\hat N(t), \int_0^t \ind_{\{\xi_s \in V^1_s\}}ds > t/2\right)\\
&\leq \Pt\left[\frac{\ind_{\{|\xi_s|<L(s)\hsl\forall s\in[0,t]\}}}{e^{\int_0^t f'(s)d\xi_s - \frac{1}{2}\int_0^t f'(s)^2 ds}}\ind_{\{\int_0^t \ind_{\{\xi_s \in V^2_s\}}ds > t/2\}}\right]\\
&\leq e^{-\frac{1}{2}\int_0^t f'(s)^2 ds + |f'(t)|L(t)+\int_0^t |f''(s)|L(s)ds}\\
&\hspace{8mm}\cdot\Pt\left(|\xi_s|<L(s)\hsl\forall s\in[0,t], \hs \int_0^t \ind_{\{\xi_s \in V^2_s\}}ds > t/2\right)\\
&\leq 2e^{-\frac{1}{2}\int_0^t f'(s)^2 ds + |f'(t)|L(t)+\int_0^t |f''(s)|L(s)ds}\\
&\hspace{8mm}\cdot\Pt\left(|\xi_s|<L(s)\hsl\forall s\in[0,t], \hs \int_0^t \ind_{\{\xi_s \in (L(s)-\delta(s+1)^{1/3},L(s))\}}ds > t/4\right).
\end{align*}
Now, by our calculation of $E$ above, the exponential part
\[2e^{-\frac{1}{2}\int_0^t f'(s)^2 ds + |f'(t)|L(t)+\int_0^t |f''(s)|L(s)ds}\]
is at most $\exp(-rt + \kappa(t+1)^{1/3})$ for some constant $\kappa$ and all large $t$. By the many-to-one theorem,
\begin{align*}
&\Pt\bigg(\exists u\in \hat N(t) : I_\alpha(u;t)< t/2\bigg)\\
&\leq e^{rt} \Pt\bigg(\xi_t \in \hat N(t), I_\alpha(\xi_t;t)< t/2\bigg)\\
&\leq e^{\kappa(t+1)^{1/3}}\Pt\left(|\xi_s|<L(s)\hsl\forall s\in[0,t], \hs \int_0^t \ind_{\{\xi_s \in (L(s)-\delta(s+1)^{1/3},L(s))\}}ds > t/4\right).
\end{align*}
We attempt to show that, for small $\delta>0$, the probability
\[\Pt\left(|\xi_s|<L(s)\hsl\forall s\in[0,t], \hs \int_0^t \ind_{\{\xi_s \in (L(s)-\delta(s+1)^{1/3},L(s))\}}ds > t/4\right)\]
is at most $\exp(-2\kappa(t+1)^{1/3})$.

For the sake of brevity we make some approximations here: for example we will use $t$ instead of $t+1$ in various places, and assume throughout that $t$ is large. Let $\tau:=\delta^2 t^{2/3}$, define
\[T_0:= \inf\{s>0: \xi_s\in(L(s)-\delta(s+1)^{1/3},L(s))\}\wedge t\]
and for $k\geq1$ let
\[T_k:= \inf\{s>T_{k-1}+\tau: \xi_s\in(L(s)-\delta(s+1)^{1/3},L(s))\}\wedge t.\]
Then for any $k\geq0$,
\begin{align*}
&\Pt(|\xi_{T_k+\tau}| < L(T_k+\tau))\\
&\leq \Pt(\xi_{T_k+\tau} - \xi_{T_k} < L(T_k+\tau) - L(T_k) + \delta(T_k+1)^{1/3})\\
&=\Pt(\xi_\tau < \gamma(T_k+\tau+1)^{1/3} - \gamma(T_k+1)^{1/3} + \delta(T_k+1)^{1/3})\\
&\leq \Pt(\xi_\tau < \gamma(\tau+1)^{1/3} + \delta(t+1)^{1/3})\\
&\approx \Pt\left(\xi_1 < \frac{\gamma t^{2/9}}{\delta^{1/3} t^{2/3}} + 1\right)
\end{align*}
which is smaller than $\Pt(\xi_1 < 2)$ when $t$ is large. We now ask how many of the $T_k$ occur strictly before $t$. We know that if
\[\int_0^t \ind_{\{\xi_s \in (L(s)-\delta(s+1)^{1/3},L(s))\}}ds > t/4\]
then
\[\sum_{k\geq1 : T_{k-1} < t} (T_k - (T_{k-1}+\tau)) \leq \frac{3t}{4}\]
and
\[\sum_{k\geq1 : T_{k-1} < t} (T_k - T_{k-1}) \geq t.\]
This tells us that
\[\sum_{k\geq1 : T_{k-1} < t} \tau \geq \frac{t}{4}\]
and hence there must be at least $t/4\tau - 1 = t^{1/3}/4\delta^2 - 1$ of the $T_k$ strictly before $t$. Let $Y$ be a binomial random variable with parameters $(\lfloor t^{1/3}/4\delta^2 - 2\rfloor, \Pt(\xi_1<2))$. At each $T_k$, the spine is within distance $\delta (t+1)^{1/3}$ of the boundary of the tube. If it jumps upwards by too much by time $T_k+\tau$, then it leaves the tube; and it has at least $\lfloor t^{1/3}/4\delta^2 - 2\rfloor$ opportunities to do so. Thus we deduce that
\begin{multline*}
\Pt\left(|\xi_s|<L(s)\hsl\forall s\in[0,t], \hs \int_0^t \ind_{\{\xi_s \in (L(s)-\delta(s+1)^{1/3},L(s))\}}ds > t/4\right)\\
\leq P( Y = 0 ) \approx (1-\Pt(\xi<2))^{t^{1/3}/4\delta^2}.
\end{multline*}
By choosing $\delta$ small we can make this smaller than $\exp(-2\kappa(t+1)^{1/3})$, which is what we required. The rest of the proof follows just as in Proposition \ref{invspineprop}.
\end{proof}

As mentioned in Section \ref{beta_sec}, Theorems \ref{betaex} and \ref{thirdex} should be compared with the work of Bramson \cite{bramson:maximal_displacement_BBM}, Lalley and Sellke \cite{lalley_sellke:conditional_limit_frontier_BBM}, Kesten \cite{kesten:BBM_with_absorption}, Hu and Shi \cite{hu_shi:minimal_pos_crit_mg_conv_BRW} and Jaffuel \cite{jaffuel:crit_barrier_brw_absorption}. Kesten \cite{kesten:BBM_with_absorption}, if translated into the language of this article, effectively considers a ``one-sided'' tube with lower boundary the critical line $\sqrt{2r}t$ and no upper boundary --- he shows that there is extinction almost surely, and that the probability of survival up to time $t$ decays like $e^{-t^{1/3}}$. If we were to consider a tube with lower boundary the line $\sqrt{2r}t$ and upper boundary $\sqrt{2r}t + \alpha t^{1/3}$ we could obtain, by the above methods, a lower bound for Kesten's asymptotic for the probability of survival up to time $t$, which would agree with Kesten's results up to a constant in the exponent. Unfortunately the corresponding upper bound, and more accurate calculations on the right-most particle in the style of Bramson \cite{bramson:maximal_displacement_BBM}, do not seem to be accessible via our current methods: the error term $E(t)$ outweighs the fine adjustments necessary to investigate such quantities. We hope to carry out further work on these and other related issues in the future.

\bibliographystyle{plain}
\def\cprime{$'$}

\end{document}